\documentclass[a4paper,11pt]{amsart}

\pdfoutput=1

\usepackage[text={440pt,660pt},centering]{geometry}

\usepackage{amsthm, amssymb, amsmath, amsfonts, mathrsfs}
\usepackage{mathtools}
\usepackage{scalerel} 
\usepackage[colorlinks=true, pdfstartview=FitV, linkcolor=blue, citecolor=blue, urlcolor=blue,pagebackref=false]{hyperref}

\usepackage{MnSymbol} 

\usepackage{esint} 

\usepackage{microtype}
\usepackage{cleveref}

\newtheorem{proposition}{Proposition}
\newtheorem{theorem}[proposition]{Theorem}
\newtheorem{lemma}[proposition]{Lemma}
\newtheorem{corollary}[proposition]{Corollary}

\theoremstyle{remark}
\newtheorem{remark}[proposition]{Remark}
\newtheorem{example}[proposition]{Example}
\theoremstyle{definition}

\numberwithin{equation}{section}
\numberwithin{proposition}{section}

\renewcommand{\le}{\leqslant}

\renewcommand{\leq}{\leqslant}
\renewcommand{\geq}{\geqslant}

\renewcommand{\subset}{\subseteq}
\newcommand{\mcl}{\mathcal}

\newcommand{\A}{{\mathcal{A}}}
\newcommand{\B}{\mathcal{B}}
\newcommand{\D}{\mathcal{D}}
\newcommand{\Diri}{\mathcal{E}}
\newcommand{\e}{\mathbf{e}}
\newcommand{\E}{\mathbb{E}}
\newcommand{\Er}{\mathbb{E}_{\rho}}
\renewcommand{\Pr}{\mathbb{P}_{\rho}}

\renewcommand{\a}{\mathbf{a}}

\renewcommand{\L}{\mathcal{L}}

\newcommand{\M}{\mathbf{M}}
\newcommand{\N}{\mathbb{N}}

\newcommand{\Ll}{\left}
\newcommand{\Rr}{\right}

\newcommand{\1}{\mathbf{1}}
\newcommand{\R}{\mathbb{R}}

\newcommand{\Zd}{{\mathbb{Z}^d}}
\renewcommand{\P}{\mathbb{P}}

\newcommand{\td}{\widetilde}
\renewcommand{\tilde}{\widetilde}

\newcommand{\de}{\delta}

\renewcommand{\d}{{\mathrm{d}}}
\newcommand{\var}{\mathbb{V}\!\mathrm{ar}}

\renewcommand{\epsilon}{\varepsilon}

\renewcommand{\fint}{\strokedint}
\newcommand{\Rd}{{\mathbb{R}^d}}
\newcommand{\n}{\mathbf{n}}
\newcommand{\mmd}{\mathcal{M}_\delta}

\newcommand{\mt}{\mathscr{M}}
\newcommand{\fil}{\mathscr{F}}

\DeclareMathOperator{\supp}{supp}

\newcommand{\Ind}[1]{\mathbf{1}_{\left\{#1\right\}}}

\renewcommand{\n}{\overrightarrow{\mathbf{n}}}
\newcommand{\op}[1]{#1^\mathrm{o}}

\newcommand{\mres}{\mathbin{\vrule height 1.4ex depth 0pt width
0.13ex\vrule height 0.13ex depth 0pt width 1ex}}

\makeatletter
\newcommand\avsuminner[2]{%
  {\sbox0{$\m@th#1\sum$}%
   \vphantom{\usebox0}%
   \ooalign{%
     \hidewidth
     \smash{\rule[.5ex]{1.5ex}{.8pt} \relax}%
     \hidewidth\cr
     $\m@th#1\sum$\cr
   }%
  }%
}
\makeatother

\title[Decay of semigroup for an infinite interacting particle system]{Decay of semigroup for an infinite interacting particle system on continuum configuration spaces}
\author{Chenlin Gu}

\address[Chenlin Gu]{DMA, Ecole normale sup\'erieure, PSL University, Paris, France}

\begin{document}

\begin{abstract}
We show the heat kernel type variance decay $t^{-\frac{d}{2}}$, up to a logarithmic correction, for the semigroup of an infinite particle system on $\Rd$, where every particle evolves following a divergence-form operator with diffusivity coefficient that depends on the local configuration of particles. The proof relies on the strategy from \cite{janvresse1999relaxation}, and generalizes the localization estimate to the continuum configuration space introduced by  S. Albeverio, Y.G. Kondratiev and M. R\"ockner.

\bigskip

\noindent \textsc{MSC 2010:} 82C22, 35B65, 60K35.

\medskip

\noindent \textsc{Keywords:} interacting particle system, heat kernel estimate, configuration space.

\end{abstract}
\maketitle
%
%
%
%
%
%
%
%
\section{Introduction}
In this work, we study an interacting diffusive particle system in $\Rd$ and the heat kernel type estimate for its semigroup. Let us give an informal introduction to the model and main result at first. We denote by $\mmd(\Rd)$ the set of point measures of type $\mu = \sum_{i=1}^{\infty} \delta_{x_i}$ on $\Rd$, which we call \textit{configurations} of particles, by $\mcl F_U$ the $\sigma$-algebra generated by $\mu(V)$ tested with all the Borel set $V \subset U$, and use the shorthand $\mcl F := \mcl F_{\Rd}$. Let $\Pr$ be the Poisson point process of density $\rho \in (0,\infty)$ as the law for the configuration $\mu$, with $\Er, \var_{\rho}$ the associated expectation and variance. We have $\a_\circ : \mmd(\Rd) \to \R^{d\times d}_{sym}$ an $\mcl F_{B_1}$-measurable symmetric matrix, i.e. it only depends on the configuration in the unit ball $B_1$, and $\vert \xi \vert^2 \leq \xi \cdot \a_\circ \xi \leq \Lambda \vert \xi \vert^2$ for any $\xi \in \Rd$. Then let $\a(\mu, x) := \a_\circ(\tau_{-x} \mu)$ be the diffusive coefficient with local interaction at $x$, where $\tau_{-x}$ represents the transport operation by the direction $-x$. Denoting by ${\mu_t := \sum_{i=1}^{\infty}\delta_{x_{i,t}} }$ the configuration at time $t \geq 0$, our model can be informally described as an infinite-dimensional system with local interaction such that every particle $x_{i,t}$ evolves as a diffusion associated to the divergence-form operator $-\nabla \cdot \a(\mu_t, x_{i,t}) \nabla$. More precisely, it is a Markov process $\Ll(\Omega, (\fil_t)_{t \geq 0}, \Pr\Rr)$ defined by the \textit{Dirichlet form}
\begin{align}\label{eq:Dirichlet}
\Diri^{\a}(f, f) := \Er\Ll[ \int_{\Rd}   \nabla f(\mu, x) \cdot \a(\mu, x) \nabla f(\mu, x) \, \d \mu(x)\Rr],
\end{align}
where the directional derivative ${\e_k \cdot \nabla f(\mu, x) := \lim_{h \to 0} \frac{1}{h} (f(\mu - \delta_x + \delta_{x + h \e_k}) - f(\mu))}$ along the canonical direction $\{\e_k\}_{1 \leq k \leq d}$ is defined for a family of suitable functions and $x \in \supp(\mu)$.

One may expect that the diffusion follows the heat kernel estimate established by the pioneering work of John Nash \cite{nash1958continuity}, as every single particle is a diffusion of divergence type. This is the object of our main theorem. Let $u : \mmd(\Rd) \to \R$ be an $\mcl F$-measurable function, depending only on the configuration in the cube $Q_{l_u} := \Ll[-\frac{l_u}{2}, \frac{l_u}{2}\Rr]^d$, and smooth with respect to the transport of every particle ( i.e. $u$ belongs to the function space $C^{\infty}_c(\mmd(\Rd))$ defined in \Cref{subsubsec:Ccinfty}), and let $u_t := \Er[u(\mu_t) \vert \fil_0]$. Denoting $L^{\infty}:= L^{\infty}(\mmd(\Rd), \mcl F, \Pr)$, we have the following estimate.  

\begin{theorem}[Decay of variance]\label{thm:main}
There exists two finite positive constants ${\gamma := \gamma(\rho, d, \Lambda)}$, ${C:=C(\rho, d, \Lambda)}$ such that for any $u \in C^{\infty}_c(\mmd(\Rd))$ supported in $Q_{l_u}$, then we have
\begin{align}\label{eq:main}
\var_{\rho}[u_t] \leq C( 1 + \vert \log t \vert )^{\gamma} \Ll(\frac{1 + l_u}{\sqrt{t}}\Rr)^d \Vert u \Vert^2_{L^{\infty}}.
\end{align} 
\end{theorem}

\smallskip

Interacting particle systems remain an active research topic, and it is hard to list all the references. We refer to the excellent monographs \cite{kipnis1998scaling, komorowski2012fluctuations, liggett2012interacting, spohn2012large} for a panorama of the field. In recent years, many works in probability and stochastic processes illustrate the diffusion universality in various models: a well-understood model is the \textit{random conductance model}, see \cite{biskup2011recent} for a survey, and especially the \textit{heat kernel bound} and \textit{invariance principle} is established for the percolation clusters in \cite{berger2007quenched, mathieu2007quenched, sidoravicius2004quenched, mathieu2008quenched, barlow2004random, hambly2009parabolic, sapozhnikov2017random}; from the view point of \textit{stochastic homogenization}, the \textit{quantitative} results are also proved in a series of work \cite{armstrong2016lipschitz, armstrong2016mesoscopic, armstrong2016quantitative, armstrong2017additive, gloria2011optimal,gloria2012optimal,gloria2014optimal,gloria2014regularity,gloria2015quantification}, and the monograph \cite{armstrong2018quantitative}, and these techiques also apply on the percolation clusters setting, as shown in \cite{armstrong2018elliptic, dario2018optimal, gu2019efficient, dario2019quantitative}; for the system of hard-spheres, Bodineau, Gallagher and Saint-Raymond prove that Brownian motion is the Boltzmann-Grad limit of a tagged particle in  \cite{bodineau2015hard, bodineau2016brownian, bodineau2018derivation}. All these works make us believe that the model in this work should also have diffusive behavior in large scale or long time.

\smallskip

Notice that our model is of \textit{non-gradient type}, and our result is established in the continuum configuration space rather than a function space on $\Rd$. In previous works, the construction of similar diffusion processes is studied by Albeverio, Kondratiev and R{\"o}ckner using Dirichlet forms in \cite{albeverio1996canonical, albeverio1996differential, albeverio1998analysis, albeverio1998analysis2}; see also the survey \cite{rockner1998stochastic}. To the best of our knowledge, we do not find \Cref{thm:main} in the literature. While in the lattice side, let us remark one important work \cite{janvresse1999relaxation} by Janvresse, Landim, Quastel and Yau, where the decay of variance is proved in the $\Zd$ zero range model, which is of \textit{gradient type}. Since our research is inspired by \cite{janvresse1999relaxation} and also uses some of their techniques, we point out our contributions in the following.

Firstly, we give an explicit bound with respect to the size of the support of the local function $u$, that is uniform over $t$; the bound $\Ll(\frac{l_u}{\sqrt{t}}\Rr)^{d}$ captures the correct typical scale. For comparison, \cite[Theorem 1.1]{janvresse1999relaxation} states the result
\begin{align}\label{eq:Longtime}
\var_{\rho}[u_t] = \frac{[\tilde{u}'(\rho)]^2 \chi(\rho)}{[8\pi \phi'(\rho)t]^{\frac{d}{2}}} + o\Ll(t^{-\frac{d}{2}}\Rr),
\end{align}
which should be considered as the asymptotic behavior in long time, and the term $o\Ll(t^{-\frac{d}{2}}\Rr)$ is of type $(l_u)^{5d}t^{-\Ll(\frac{d}{2}+\epsilon\Rr)}$ if one tracks carefully the dependence of $l_u$ in the steps of the proof of \cite[Theorem 1.1]{janvresse1999relaxation}. To get the typical scale $\Ll(\frac{l_u}{\sqrt{t}}\Rr)^{d}$, we do some combinatorial improvement in the intermediate coarse-graining argument in \cref{eq:PerTwoa}; see also \Cref{fig:coarse-graining} for illustration. On the other hand, we also wonder if we could establish a similar result as \cref{eq:Longtime} to identify the diffusive constant in the long time behavior. This an interesting question and one perspective in future research, but a major difficulty here is to characterize the effective diffusion constant, because the zero range model satisfies the \textit{gradient condition} while our model does not. We believe that it is related to the \textit{bulk diffusion coefficient} and the equilibrium density fluctuation in the lattice nongradient model as indicated in \cite[eq.(2.14), Proposition 2.1]{spohn2012large}.  

Secondly, we extend a localization estimate to the continuum configuration space: under the same context of \Cref{thm:main}, and recalling that $\mcl F_{Q_K}$ represents the information of $\mu$ in the cube ${Q_K = \Ll[-\frac{K}{2}, \frac{K}{2}\Rr]^d}$, we define $\A_K u_t := \Er[u_t \vert \mcl F_{Q_K}]$, and show that for every ${t \geq \max\Ll\{(l_u)^2, 16 \Lambda^2 \Rr\}}$ and $K \geq \sqrt{t}$ 
\begin{equation}\label{eq:LocalIntro}
\Er\Ll[(u_t - \A_K u_t)^2\Rr] \leq C(\Lambda)\exp\left(- \frac{K}{\sqrt{t}} \right)\Er\Ll[u^2\Rr].
\end{equation}
This is a key estimate appearing in \cite[Proposition 3.1]{janvresse1999relaxation}, and is also natural as $\sqrt{t}$ is the typical scale of diffusion, thus when $K \gg \sqrt{t}$ one get very good approximation in \cref{eq:LocalIntro}. Its generalization in the continuum configuration space is non-trivial, since in the proof of \cite[Proposition 3.1]{janvresse1999relaxation}, one tests the Dirichlet form with $\A_K u_t$, but in our model it is not in the domain of Dirichlet form $\D(\Diri^{\a})$ and one cannot put $\A_K u_t$ directly in the Dirichlet form \cref{eq:Dirichlet}. This is one essential difference between our model and a lattice model. To solve it, we have to apply some regularization steps which we present in \Cref{thm:localization}.

Finally, we remark kindly a minor error in the proof in \cite{janvresse1999relaxation} and fix it when revisiting the paper. This will be presented in \Cref{subsec:Outline} and \Cref{rmk:Error}.

\smallskip

The rest of this article is organized as follows. In \Cref{sec:Pre}, we define all the notations and the rigorous construction of our model. \Cref{sec:Strategy} is the main part of the proof of \Cref{thm:main}, where \Cref{subsec:Outline} gives its outline and we fix the minor error in \cite{janvresse1999relaxation} mentioned above. The proof of some technical estimates used in \Cref{sec:Strategy} are put in the last two sections, where \Cref{sec:Localization} proves the localization estimate \cref{eq:LocalIntro} in continuum configuration space, and \Cref{sec:Toolbox} serves as a toolbox of other estimates including spectral inequality, perturbation estimate and calculation of the entropy.

\section{Preliminaries}\label{sec:Pre}
\subsection{Notations}\label{subsec:Notations}
In this part, we introduce the notations used in this paper. We write~$\Rd$ for the $d$-dimensional Euclidean space, $B_r(x)$ for the ball of radius $r$ centered at $x$, and $Q_s(x) := x + \Ll[- \frac{s}{2}, \frac{s}{2}\Rr]^d$ as the cube of edge length $s$ centered at $x$. We also denote by $B_r$ and $Q_s$ respectively short for $B_r(0)$ and $Q_s(0)$. The lattice set is defined by $ \mcl Z_s := \Zd \cap Q_s$.
\subsubsection{Continuum configuration space}
For any metric space $(E,d)$, we denote by $\mcl M(E)$ the set of Radon measures on $E$. For every Borel set $U \subset E$, we denote by $\mcl F_U$ the smallest {$\sigma$-algebra} such that for every Borel subset $V \subset U$, the mapping ${\mu \in \mcl M(E) \mapsto \mu(V)}$ is measurable. For a $\mcl F_U$-measurable function $f : \mcl M(E) \to \R$, we say that $f$ supported in $U$ i.e. $\supp(f) \subset U$. In the case $\mu \in \mcl M(E)$ is of finite total mass, we write
\begin{equation}  
\label{e.def.fint}
\fint f \, \d \mu := \frac{\int f \, \d \mu}{\int  \d \mu}.
\end{equation}

We also define the collection of point measure $\mmd(E) \subset \mcl M(E)$
\begin{align*}
\mmd (E) := \Ll\{ \mu \in \mcl M(E) :  \mu = \sum_{i \in I} \delta_{x_i} \text{ for some } I \text{ finite or countable}, \text{ and } x_i \in E \text{ for any } i \in I \Rr\},
\end{align*}
which serves as the \textit{continuum configuration space} where each Dirac measure stands the position of a particle. In this work we will mainly focus on the Euclidean space $\Rd$ and its associated point measure space $\mmd(\Rd)$, and use the shorthand notation $\mcl F := \mcl F_{\Rd}$.

We define two operations for elements in $\mmd(\Rd)$: \textit{restriction} and  \textit{transport}.
\begin{itemize}
\item For every $\mu \in \mmd(\Rd)$ and Borel set $U \subset \Rd$, we define the restriction operation $\mu \mres U$, such that for every Borel set $V \subset \Rd$, $(\mu \mres U) (V) = \mu (U \cap V)$. Then for a function $f : \mmd(\Rd) \to \R$ which is $\mcl F_U$-measurable, we have $f(\mu) = f(\mu \mres U)$.
\item The transport on the set is defined as 
\begin{align*}
\forall h \in \Rd, U \subset \Rd,  \tau_h U := \{y + h: y \in U\}.
\end{align*}
Then for every $\mu \in \mmd(\Rd)$ and $h \in \Rd$, we define the transport operation $\tau_h \mu$ such that for every Borel set $U$, we have 
\begin{align}\label{eq:transportMu}
\tau_h \mu (U) := \mu(\tau_{-h} U).
\end{align} 
For $f$ an $\mcl F_V$-measurable function,  we also define the transport operation $\tau_h f$ as a pullback that 
\begin{align}\label{eq:transportFunction}
\tau_h f(\mu) := f(\tau_{-h} \mu),
\end{align} 
which is an $\mcl F_{\tau_h V}$-measurable function.
\end{itemize}
Notice that the restriction operation can be defined similarly in $\mcl M(E)$ for a metric space, but the transport operation requires that $E$ is at least a vector space.

\smallskip

We fix $\rho > 0$ once and for all, and define $\Pr$ a probability measure on $(\mmd(\Rd),\mcl F)$, to be the Poisson measure on $\Rd$ with density $\rho$ (see \cite{kingman2005p}). We denote by $\Er$ the expectation, $\var_{\rho}$ the variance associated with the law~$\Pr$, and by $\mu$ the canonical $\mmd(\Rd)$-valued random variable on the probability space $(\mmd(\Rd), \mcl F, \Pr)$. In the case $U \subset \Rd$ a bounded Borel set and $f$ a $\mcl F_{U}$-measurable function, we can rewrite the expectation $\Er[f]$ in an explicit expression
\begin{equation}  
\label{e.bulky}
\Er \Ll[ f \Rr] = \sum_{N = 0}^{+\infty} e^{-\rho|U|} \frac{(\rho |U|)^N}{N!} \fint_{U^N} f \Ll( \sum_{i = 1}^N \de_{x_i} \Rr)   \, \d x_1 \cdots \d x_N. 
\end{equation}
For instance, for every bounded Borel set $U \subset \Rd$ and bounded measurable function $g : U \to \R$, we can write
\begin{equation*}  
\Er \Ll[ \int_U g(x) \, \d \mu(x) \Rr] = \rho \int_U g(x) \, \d x.
\end{equation*}
Notice that the measure $\mu$ is a Poisson point process under $\Pr$. In particular, the measures $\mu \mres U$ and $\mu \mres (\Rd \setminus U)$ are independent, and the conditional expectation $\Er \Ll[ \cdot \vert \mcl F_{(\Rd \setminus U)}\Rr]$ can thus be described equivalently as an averaging over the law of~$\mu \mres U$.

For any $1 \leq p < \infty$, we denote by $L^p(\mmd(U))$ the set of $\mcl F_U$-measurable functions $f : \mmd(U) \to \R$ such that the norm
\begin{equation*}  
\|f \|_{L^p(\mmd(U))} :=  \Ll(\Er \Ll[ |f|^p \Rr] \Rr)^\frac 1 p 
\end{equation*}
is finite and $L^p$ short for $L^p(\mmd(\Rd))$. We denote by $L^{\infty}(\mmd(U))$ the norm defined by essential upper bound under $\Pr$.

\subsubsection{Derivative and $C_c^{\infty}(\mmd(U))$}\label{subsubsec:Ccinfty}
We define the directional derivative for a $\mcl F_{U}$-measurable function $f : \mmd(U) \to \R$. Let $\{\e_k\}_{1 \leq k \leq n}$ be $d$ canonical directions, for $x \in \supp(\mu)$, we define 
\begin{align*}
\partial_k f(\mu, x) := \lim_{h \to 0} \frac{1}{h} (f(\mu - \delta_x + \delta_{x + h \e_k}) - f(\mu)),
\end{align*}
if the limit exists, and the gradient as a vector 
\begin{align*}
{\nabla f(\mu, x) := (\partial_1 f(\mu, x), \partial_2 f(\mu, x), \cdots \partial_d f(\mu, x))}.
\end{align*}
One can define the function with higher derivative iteratively, but  here we use a more natural way: for every Borel set $U \subset \Rd$ and $N \in \N$,  let $\mmd(U, N) \subset \mmd(E)$ be defined as
\begin{align*}
\mmd(U, N) := \Ll\{\mu \in \mmd(\Rd) : \mu = \sum_{i=1}^N x_i,  x_i \in U \text{ for every } 1 \leq i \leq N \Rr\}.
\end{align*} 
Then a function $f : \mcl M_\de(U,N) \to \R$ can be identified with a function $\td f : U^N \to \R$ by setting
\begin{equation}  
\label{e.def.tdf}
\td f(x) = \td f(x_1,\ldots,x_N) := f \Ll( \sum_{i = 1}^N \de_{x_i} \Rr) .
\end{equation}
The function $\td f$ is invariant under permutations of its $N$ coordinates. Conversely, any function satisfying this symmetry can be identified with a function from $\mmd(U,N)$ to $\R$. We denote by $C^\infty(\mmd(U,N))$ the set of functions $f : \mmd(U,N) \to \R$ such that $\td f$ is infinitely differentiable. 
 For every $f \in C^\infty(\mmd(U,N))$ and $x_1,\ldots,x_N \in U$, the gradient at $x_1$ coincides with the its canonical sense for the coordinate $x_1$.
\begin{equation}\label{def.grad.mu}
\nabla f \Ll( \sum_{i = 1}^N \de_{x_i} ,x_1 \Rr) = \nabla_{x_1} \td f(x_1,\ldots,x_N).
\end{equation}

We denote by $C^\infty_c(\mmd(U))$ the set of functions $f : \mmd(U) \to \R$ that satisfy:
\begin{enumerate}  
\item there exists a compact Borel set $V \subset U$ such that $f$ is $\mcl F_V$-measurable; 
\item for every $N \in \N$,
\begin{equation*}  
\text{the mapping } \Ll\{
\begin{array}{rcl}  
\mmd(U,N) & \to & \R \\
\mu & \mapsto & f(\mu)
\end{array}
\Rr.
\mbox{belongs to $C^\infty(\mmd(U,N))$.}
\end{equation*}
\item the function is bounded. 
\end{enumerate}
A more heuristic description for $f \in C^\infty_c(\mmd(U))$ is a  function uniformly bounded, depending only on the information in a compact subset $V \subset U$, and when we do projection $f(\mu) = f(\mu \mres V)$ it can be identified as a function $C^{\infty}$ with finite coordinate, and also smooth when the number of particles in $V$ changes.  

\subsubsection{Sobolev space on $\mmd(U)$}
We define the $H^1(\mmd(U))$ norm by
\begin{equation*}  
\|f\|_{H^1(\mmd(U))} := \Ll( \|f\|_{L^2(\mmd(U))}^2 + \Er \Ll[ \int_U |\nabla f|^2 \, \d \mu \Rr] \Rr)^\frac 1 2,
\end{equation*}
and let $H^1_0(\mmd(U))$ denote the completion with respect to this norm of the space 
\begin{equation*}  
\Ll\{f \in C^\infty_c(\mmd(U)) \ : \ \|f\|_{H^1(\mmd(U))} < \infty\Rr\}.
\end{equation*}

\subsection{Construction of model}\label{subsec:Construction}
\subsubsection{Diffusion coefficient}
In this part, we define the coefficient field of the diffusion. We give ourselves a symmetric matrix valued  function $\a_\circ : \mmd(\Rd) \to \R^{d \times d}_{sym}$ which satisfies the following properties:
\begin{itemize}  
\item uniform ellipticity: there exists $\Lambda \in [1,+\infty)$ such that for every $\mu \in \mmd(\Rd)$ and every $\xi \in \Rd$,  
\begin{equation}  
\label{e.ellipticity}
\vert \xi \vert^2 \le \xi \cdot \a_\circ(\mu) \xi \le \Lambda \vert \xi \vert^2\, ;
\end{equation}
\item locality: for every $\mu \in \mmd(\Rd)$, $\a_\circ(\mu) = \a_\circ\Ll(\mu \mres B_1\Rr)$.
\end{itemize}
We extend $\a_\circ$ by stationarity using the transport operation defined in \cref{eq:transportFunction}: for every $\mu \in \mmd(\Rd)$ and $x \in \Rd$,
\begin{equation*}  
\a(\mu,x) := \tau_x \a_{\circ}(\mu) = \a_\circ(\tau_{-x}\mu).
\end{equation*}
A typical example of a coefficient field $\a$ of interest is ${\a_\circ(\mu) := (1 + \1_{\{ \mu(B_1) = 1  \}}})\mathbf{Id}$
whose extension is given by $\a(\mu,x) := (1 + \1_{\{ \mu(B_1(x)) = 1  \}}) \mathbf{Id}$.
In words, for $x \in \supp(\mu)$, the quantity $\a(\mu,x)$ is equal to $2$ whenever there is no other point than $x$ in the unit ball around $x$, and is equal to $1$ otherwise.

\subsubsection{Markov process defined by Dirichlet form}
In this part, we construct our infinite particle system on $\mmd(\Rd)$ by Dirichlet form (see \cite{fukushima2010dirichlet, ma2012introduction} for the notations). We define at first the non-negative bilinear symmetric form 
\begin{align*}
\Diri^{\a}(f, g) := \Er\Ll[ \int_{\Rd}  \nabla f(\mu, x) \cdot \a(\mu, x) \nabla g(\mu, x) \, \d \mu(x)\Rr],
\end{align*}
on its \textit{domain} $\D(\Diri^{\a})$ that 
\begin{align*}
\D(\Diri^{\a}) := H^1_0(\mmd(\Rd)).
\end{align*}
We also use $\Diri^{\a}(f) := \Diri^{\a}(f, f)$ for short. It is clear that $\Diri^{\a}$ is \textit{closed} and \textit{Markovian} thus it is a \textit{Dirichlet form}, so it defines the correspondence between the Dirichlet form and the generator $\L$ that 
\begin{align*}
\Diri^{\a}(f,g) = \Er\Ll[f (-\L)g\Rr], \qquad \D(\Diri^{\a}) = \D(-\L).
\end{align*}
and a $L^2(\mmd(\Rd))$ strongly continuous Markov semigroup $(P_t)_{t \geq 0}$. We denote by $(\fil_t)_{t \geq 0}$ its filtration and $(\mu_t)_{t \geq 0}$ the associated $\mmd(\Rd)$-valued Markov process which stands the configuration of the particles, then for any $u \in L^2(\mmd(\Rd))$, 
\begin{align*}
u_t(\mu) := P_t u(\mu) = \Er[u(\mu_t) \vert \fil_0],
\end{align*}
is an element in $\D(\Diri^{\a})$ and is characterized by the parabolic equation on $\mmd(\Rd)$ that for any $v \in \D(\Diri^{\a})$
\begin{align}\label{eq:Evolution}
\Er[u_t v] - \Er[u v] = - \int_{0}^t \Diri^{\a}(u_s, v) \, \d s.
\end{align}
Finally, we remark that the average is conserved for $u_t$ as we test \cref{eq:Evolution} by constant $1$ that 
\begin{align}\label{eq:uAverage}
\Er[u_t] - \Er[u] = - \int_{0}^t \Er\Ll[ \int_{\Rd} \nabla 1 \cdot  \a(\mu, x) \nabla u_s(\mu, x) \, \d \mu\Rr] \, \d s = 0.
\end{align}
In this work, we focus more on the quantitative property of $P_t$; see \cite{rockner1998stochastic} for more details about the trajectory property of similar type of process. 

\subsection{A solvable case}\label{subsec:Solvable}
We propose a solvable model to illustrate that the behavior of this process is close to the diffusion and the rate of decay is the best one that we can expect.

In the following, we suppose that $\a = \frac{1}{2}$ which means that in fact every particle evolves as a Brownian motion i.e. $\mu = \sum_{i=1}^{\infty} \delta_{x_i} $, $\mu_t = \sum_{i=1}^{\infty} \delta_{B^{(i)}_t}$ that $\Ll(B^{(i)}_t\Rr)_{t \geq 0}$ is a Brownian motion issued from $x_i$ and $\Ll(B^{(i)}_{\cdot}\Rr)_{i\in \N}$ is independent.

\begin{example}
Let $u(\mu) := \int_{\Rd} f \, \d \mu $ with $f \in C^{\infty}_c(\Rd)$. 
In this case, we have
\begin{align*}
u_t(\mu) = P_t u(\mu) &= \Er\Ll[ u(\mu_t) \vert \fil_0 \Rr] =  \Er\Ll[ \Ll. \sum_{i \in \N} f\Ll(B^{(i)}_t\Rr) \Rr\vert \fil_0 \Rr] =  \int_{\Rd} f_t(x) \, \d \mu(x),
\end{align*}  
where $f_t \in C^{\infty}(\Rd)$ is the solution of the Cauchy problem of the standard heat equation: ${\Phi_t(x) = \frac{1}{(2\pi t)^{\frac{d}{2}}} \exp\Ll(- \frac{\vert x \vert^2}{2t}\Rr)}$ and $f_t(x) = \Phi_t \star f(x)$. Then we use the formula of variance for Poisson process
\begin{align*}
\var_{\rho}\Ll[u \Rr] &= \rho\int_{\Rd} f^2(x) \, \d x = \rho\Vert f \Vert^2_{L^2\Ll(\Rd\Rr)}, \\
\var_{\rho}\Ll[u_t \Rr] &= \rho\int_{\Rd} f_t^2(x) \, \d x = \rho\Vert f_t \Vert^2_{L^2\Ll(\Rd\Rr)}.
\end{align*}
By the heat kernel estimate for the standard heat equation, we known that ${\Vert f_t \Vert^2_{L^2(\Rd)} \simeq C(d)t^{-\frac{d}{2}}  \Vert f \Vert^2_{L^1(\Rd)}}$, thus the scale $t^{-\frac{d}{2}}$ is the best one that we can obtain. Moreover, if we take $f = \Ind{Q_r}$, and ${t=r^{2(1-\epsilon)}}$ for a small $\epsilon >0$, then we see that the typical scale of diffusion is a ball of size $r^{1-\epsilon}$. So for every $x \in Q_{r\Ll(1-r^{-\frac{\epsilon}{2}}\Rr)}$, the value $f_t(x) \simeq 1 - e^{-r^{\frac{\epsilon}{2}}}$ and we have 
\begin{align*}
\var_{\rho}\Ll[u_t \Rr] = \rho\int_{\Rd} f_t^2(x) \, \d x \geq \rho r^d(1-r^{-\frac{\epsilon}{2}}) = (1-r^{-\frac{\epsilon}{2}})\var_{\rho}\Ll[u \Rr].
\end{align*}
It illustrates that before the scale $t = r^2$, the decay is very slow so in the \Cref{thm:main} the factor $\Ll(\frac{l_u}{\sqrt{t}}\Rr)^d$ is reasonable.
\end{example}

\section{Strategy of proof}\label{sec:Strategy}
In this part, we state the strategy of the proof of \Cref{thm:main}. We will give a short outline in \Cref{subsec:Outline}, which can be see as an ``approximation-variance decomposition", and then focus on the term approximation in \Cref{subsec:Approximation}. Several technical estimates will be used in this procedure and their proofs will be postponed in \Cref{sec:Localization} and \Cref{sec:Toolbox}.
\subsection{Outline}\label{subsec:Outline}
As mentionned, this work is inspired from  \cite{janvresse1999relaxation}, and we revisit the strategy here.
We pick a centered $u \in C^{\infty}_c(\mmd(\Rd))$ supported in $Q_{l_u}$ such that $\Er[u]  = 0$ and this implies $\Er[u_t] = 0$ from \cref{eq:uAverage}. Then we set a multi-scale $\{t_n\}_{n\geq 0}, t_{n+1} = Rt_n$, where $R > 1$ is a scale factor to be fixed later. It suffices to prove that \cref{eq:main} for every $t_n$, then for $t \in [t_n, t_{n+1}]$, one can use the decay of $L^2$ that 
\begin{align*}
\Er[(u_t)^2]  \leq \Er[(u_{t_n})^2]  \leq C(1+\log(t_n))^{\gamma}\Ll(\frac{1+l_u}{\sqrt{t_n}}\Rr)^d \Vert u \Vert^2_{L^{\infty}} \leq C R^{\frac{d}{2}}(1 + \log t)^{\gamma}\Ll(\frac{1 + l_u}{\sqrt{t}}\Rr)^d \Vert u \Vert^2_{L^{\infty}},
\end{align*}
then by resetting the constant $C$ one concludes the main theorem. Another ingredient of the proof is an ``approximation-variance type decomposition": 
\begin{equation}\label{eq:defTwoTerms}
\begin{split}
u_t &= v_t + w_t, \\
v_t &:= u_t - \frac{1}{\vert \mcl Z_K \vert} \sum_{y \in \mcl Z_K} \tau_{y} u_t,\\
w_t &:= \frac{1}{\vert \mcl Z_K \vert} \sum_{y \in \mcl Z_K} \tau_{y} u_t,
\end{split}
\end{equation}
where we recall $\mcl Z_K = Q_K \cap \Zd$ is the lattice set of scale $K$. The philosophy of this decomposition is that in long time, the information in a local scale $K$ is mixed, thus $w_t$ as a spatial average is a good approximation of $u_t$ and $v_t$ is the error term. Thus, the following control \Cref{prop:Variance} and \Cref{prop:Perturbation} of the two terms $w_t$ and $v_t$ proves the main theorem \Cref{thm:main}.

\begin{proposition}\label{prop:Variance}
There exists a finite positive number $C := C(d)$ such that for any ${u \in C^{\infty}_c(\mmd(\Rd))}$ with mean zero, supported in $Q_{l_u}$ and for any $K \geq l_u$, we have 
\begin{align}\label{eq:Variance}
\var_{\rho}\left[ \left( \frac{1}{\vert \mcl Z_K \vert} \sum_{y \in \mcl Z_K} \tau_{y} u_t \right)^2\right] \leq  C(d)\Ll(\frac{l_u}{K}\Rr)^d \Er[u^2].
\end{align}   
\end{proposition}
\begin{proof}
Then we can estimate the variance simply by $L^2$ decay that 
\begin{multline*}
\Er[(w_t)^2] = \Er\Ll[\Ll( P_t\Ll(\frac{1}{\vert \mcl Z_K \vert} \sum_{y \in \mcl Z_K}  \tau_{y}u \Rr)\Rr)^2\Rr] 
\leq \Er\Ll[\Ll(\frac{1}{\vert \mcl Z_K \vert} \sum_{y \in \mcl Z_K}  \tau_{y}u \Rr)^2\Rr] 
= \frac{1}{\vert \mcl Z_K \vert^2}  \sum_{x, y \in \mcl Z_K}  \Er\Ll[ (\tau_{x-y} u) u\Rr].
\end{multline*}
We know that for $\vert x-y \vert \geq l_u$, then the term $ \tau_{x-y}u$ and $u$ is independent so $ \Er\Ll[ (\tau_{x-y} u) u\Rr]$ = 0. This concludes \cref{eq:Variance}.
\end{proof}

\begin{proposition}\label{prop:Approximation}
There exists two finite positive numbers $C := C(d, \rho), \gamma := \gamma(d, \rho)$ such that for any ${u \in C^{\infty}_c(\mmd(\Rd))}$ supported in $Q_{l_u}$, $K \geq l_u$ and $v_t$ defined in \cref{eq:defTwoTerms}, for ${t_n \geq \max\Ll\{l_u^2, 16 \Lambda^2 \Rr\}}, t_{n+1} = Rt_n$ with $R>1$ we have 
\begin{align}\label{eq:Approximation}
(t_{n+1})^{\frac{d+2}{2}} \Er[(v_{t_{n+1}})^2] - (t_n)^{\frac{d+2}{2}} \Er[(v_{t_n})^2] \leq C (\log(t_{n+1}))^{\gamma} K^2 (l_u)^{d} \Vert u\Vert^2_{L^{\infty}} + \Er[u^2].
\end{align} 
\end{proposition}
\begin{proof}[Proof of \Cref{thm:main} from \Cref{prop:Approximation} and \Cref{prop:Variance}]
For the case $t \leq (l_u)^2$ or ${t \leq 16 \Lambda^2}$, the right hand side of \cref{eq:main} is larger than $\Er[u^2]$ and we can use the $L^2$ decay to prove the theorem. Thus without loss of generality, we set ${t_0 := \max\Ll\{l_u^2, 16 \Lambda^2 \Rr\}}$ and put \cref{eq:Approximation} into \cref{eq:main} by setting $K:= \sqrt{t_{n+1}}$ that
\begin{equation}\label{eq:FixError}
\begin{split}
 &\Er[(u_{t_{n+1}})^2] \\
\leq & 2\Er[(v_{t_{n+1}})^2] + 2\Er[(w_{t_{n+1}})^2] \\
\leq & 2 \left(\frac{t_n}{t_{n+1}}\right)^{\frac{d+2}{2}}\Er[(v_{t_{n}})^2] + 2 (t_{n+1})^{-\frac{d+2}{2}}\left(C (\log(t_{n+1}))^{\gamma} t_{n+1} (l_u)^{d} \Vert u\Vert^2_{L^{\infty}} + \Er[u^2]\right) \\
& \qquad \qquad + 2\Er[(w_{t_{n+1}})^2]\\
\leq & 4 \left(\frac{t_n}{t_{n+1}}\right)^{\frac{d+2}{2}}\Er[(u_{t_{n}})^2] +  2 (t_{n+1})^{-\frac{d+2}{2}}\left(C (\log(t_{n+1}))^{\gamma} t_{n+1} (l_u)^{d} \Vert u\Vert^2_{L^{\infty}} + \Er[u^2]\right)\\
& \qquad + 4 \left(\frac{t_n}{t_{n+1}}\right)^{\frac{d+2}{2}}\Er[(w_{t_{n}})^2] + 2\Er[(w_{t_{n+1}})^2].
\end{split}
\end{equation}
We set $U_n = (t_n)^{\frac{d}{2}}\Er[(u_{t_{n}})^2]$ and put
\cref{eq:Variance} into the equation above, we have 
\begin{align*}
U_{n+1} \leq \theta U_{n} + C_2 \Ll((\log(t_{n+1}))^{\gamma} (l_u)^{d} \Vert u\Vert^2_{L^{\infty}} + (t_{n+1})^{-1} \Er[u^2]\Rr) + C_3 (l_u)^{d}\Er[u^2] ,
\end{align*}
where $\theta = 4R^{-1}$. By choosing $R$ large such that $\theta \in (0,1)$, we do a iteration for the equation above to obtain that 
\begin{align*}
U_{n+1} &\leq \sum_{k=1}^n\Ll(C_2 \Ll((\log(t_{n+1}))^{\gamma} (l_u)^{d} \Vert u\Vert^2_{L^{\infty}} +  \Er[u^2]\Rr) + C_3 (l_u)^{d}\Er[u^2]\Rr)\theta^{n-k} + U_0 \theta^{n+1} \\
&\leq \frac{1}{1-\theta}\Ll(C_2 \Ll((\log(t_{n+1}))^{\gamma} (l_u)^{d} \Vert u\Vert^2_{L^{\infty}} +  \Er[u^2]\Rr) + C_3 (l_u)^{d}\Er[u^2]\Rr) + (l_u)^{d}\Er[u^2] \\
\Longrightarrow  & \Er[(u_{t_{n+1}})^2] \leq C_4( \log(t_{n+1}) )^{\gamma}\Ll(\frac{l_u}{\sqrt{t_{n+1}}}\Rr)^d \Vert u \Vert^2_{L^{\infty}}.
\end{align*}
\end{proof}
\begin{remark}\label{rmk:Error}
We remark that there is a small error in the similar argument in \cite[Proof of Proposition 2.2]{janvresse1999relaxation}: the authors apply \cref{eq:Approximation} from $t_0$ to $t_n$, and they neglect the change of scale in $K$ at the endpoints $\{t_{n}\}_{n \geq 0}$. However, it does not harm the whole proof and we fix it here: we add one more step of decomposition in \cref{eq:FixError}, and put the iteration directly in $u_t$ instead of $v_t$, which avoids the problem of the changes of $K$.
\end{remark}

\subsection{Error for the approximation}\label{subsec:Approximation}
In this part, we prove \Cref{prop:Approximation}. The proof can be divided into 6 steps.

\begin{proof}[Proof of \Cref{prop:Approximation}].
\textit{Step 1: Setting up.}
To shorten the equation, we define
\begin{align}
\Delta_n := (t_{n+1})^{\frac{d+2}{2}} \Er[(v_{t_{n+1}})^2] - (t_n)^{\frac{d+2}{2}} \Er[(v_{t_n})^2],
\end{align}
and it is the goal of the whole subsection. In the step setting up, we do derivative for the flow $t^{\frac{d+2}{2}}\Er[(v_t)^2]$ that 
\begin{align}\label{eq:ApproxDerivative}
\Delta_n = \int_{t_n}^{t_{n+1}}\left(\frac{d+2}{2}\right)t^{\frac{d}{2}}\Er[(v_t)^2] - 2t^{\frac{d+2}{2}}\Er[v_t (-\L v_t)] \, \d t.
\end{align}

\textit{Step 2: Localization.}
We set $\A_L v_t := \E\Ll[v_t | \mcl F_{Q_L} \Rr]$ and use it to approximate $v_t$ in $L^2$. Since it is a diffusion process, one can guess naturally a scale larger than $\sqrt{t}$ will have enough information for this approximation. In \Cref{thm:localization} we prove an estimate 
\begin{align*}
\Er\Ll[(v_t - \A_L v_t)^2\Rr] \leq C(\Lambda)\exp\Ll(-\frac{L}{\sqrt{t}}\Rr)\Er\Ll[(v_0)^2\Rr],
\end{align*}
and we choose $L = \lfloor \gamma \log(t_{n+1})\rfloor \sqrt{t_{n+1}}, \gamma > \frac{d+4}{2}$ here, and put it back to \cref{eq:ApproxDerivative} to obtain 
\begin{equation}\label{eq:ApproxLocalization}
\begin{split}
\Delta_n \leq &\int_{t_n}^{t_{n+1}}(d+2)t^{\frac{d}{2}}\Er\Ll[\Ll(\A_L v_t\Rr)^2\Rr]+ (d+2)t^{\frac{d}{2}-\gamma} \Er\Ll[(v_0)^2\Rr] - 2t^{\frac{d+2}{2}}\Er[v_t (-\L v_t)] \, \d t \\
\leq & \Er[(u_0)^2] + \int_{t_n}^{t_{n+1}}(d+2)t^{\frac{d}{2}}\Er\Ll[\Ll(\A_L v_t\Rr)^2\Rr] - 2t^{\frac{d+2}{2}}\Er[v_t (-\L v_t)] \, \d t . 
\end{split}
\end{equation}

\textit{Step 3: Approximation by density.}
We apply a second approximation: we choose another scale $l>0$, whose value will be fixed but $L/l \in \mathbb{N}$ and $l \simeq \sqrt{t_{n+1}}$. We denote by $q = (L/l)^d$ and $\M_{L,l} = (\M_1, \M_2 \cdots \M_q)$ a random vector, where $\M_i$ is the number of the particle in $i$-th cube of scale $l$. Then we define an operator 
\begin{align*}
\B_{L,l} v_t := \Er\Ll[v_t \vert \M_{L,l} \Rr].
\end{align*}
The main idea here is that the random vector $\M_{L,l}$ captures the information of convergence, once we know the density in every cube of scale $l \simeq \sqrt{t_{n+1}}$ converges to $\rho$. In \Cref{prop:AB} we will prove a spectral inequality that 
\begin{align*}
\Er\Ll[(\A_L v_t - \B_{L,l} v_t)^2\Rr] \leq R_0 l^2 \Er\Ll[v_t (-\L v_t)\Rr].
\end{align*}
We put this estimate into \cref{eq:ApproxLocalization} 
\begin{align*}
\Delta_n \leq & \Er[(u_0)^2] + \int_{t_n}^{t_{n+1}} 2(d+2)t^{\frac{d}{2}}\Er\Ll[\Ll(\B_{L,l} v_t\Rr)^2\Rr] + 2t^{\frac{d}{2}}((d+2)R_0 l^2 -t)\Er[v_t (-\L v_t)] \, \d t \\
\leq & \Er[(u_0)^2] + \int_{t_n}^{t_{n+1}} 2(d+2)t^{\frac{d}{2}}\Er\Ll[\Ll(\B_{L,l} v_t\Rr)^2\Rr] \, \d t, 
\end{align*}
where we obtain the last line by choosing a scale $l = c \sqrt{t_{n+1}}$ such that $(d+2)R_0 l^2 \leq t_n $ and $L/l \in \mathbb{N}$. 

It remains to estimate how small $\Er\Ll[\Ll(\B_{L,l} v_t\Rr)^2\Rr]$ is. The typical case is that the density is close to $\rho$ in every cube of scale $l$ in $Q_L$. Let us define $M = (M_1, M_2, \cdots M_q)$, and we have 
\begin{align*}
\B_{L,l} v_t(M) = \Er\Ll[v_t \vert \M_{L,l} = M \Rr]. 
\end{align*}
Then we call $\mcl C_{L,l,\rho,\delta}$ the $\delta$-good configuration that 
\begin{align}
\mcl C_{L,l,\rho,\delta} := \Ll\{M \in \mathbb{N}^{q} \Ll\vert \forall 1 \leq i\leq q, \Ll\vert\frac{M_i}{\rho\vert Q_l \vert} - 1\Rr\vert \leq \delta \Rr.\Rr\}.
\end{align}
We can use standard Chernoff bound and union bound to prove the upper bound of $\Pr\Ll[\M_{L,l} \notin \mcl C_{L,l,\rho,\delta} \Rr]$: for any $\lambda > 0$, we have 
\begin{align*}
\Pr\Ll[\exists \leq i\leq q, \frac{\M_i}{\rho\vert Q_l \vert}  \geq 1+\delta \Rr] &\leq \Ll(\frac{L}{l}\Rr)^d \exp(-\lambda(1+\delta))\Er\Ll[\exp\Ll(\frac{\lambda\mu(Q_l)}{\rho \vert Q_l\vert}\Rr)\Rr] \\
&=\Ll(\frac{L}{l}\Rr)^d \exp\Ll(-\lambda(1+\delta)+ \rho \vert Q_l\vert \Ll(e^{\frac{\lambda }{\rho \vert Q_l\vert}}-1\Rr)\Rr) \\
&\leq \Ll(\frac{L}{l}\Rr)^d \exp\Ll(-\lambda \delta + \frac{\lambda^2}{\rho \vert Q_l\vert}\Rr).
\end{align*}
In the second line we use the exact Laplace transform for $\mu(Q_l)$ as we know ${\mu(Q_l) \stackrel{\text{law}}{\sim} \text{Poisson}(\rho \vert Q_l \vert)}$. Then we do optimization by choosing $\lambda = \frac{\delta \rho \vert Q_l \vert}{2}$. The other side is similar and we conclude
\begin{align}
\Pr\Ll[\M_{L,l} \notin \mcl C_{L,l,\rho,\delta} \Rr] \leq \Ll(\gamma \log(t_{n+1})\Rr)^d \exp\Ll(-\frac{\rho \vert Q_l\vert \delta^2}{4}\Rr).
\end{align}
For the case $M \notin \mcl C_{L,l,\rho,\delta}$, we can bound $\B_{L,l} v_t(M)$ naively by $\vert \B_{L,l} v_t(M) \vert \leq C\Vert u_0\Vert_{L^{\infty}}$, thus we have 
\begin{align*}
\Er\Ll[\Ll(\B_{L,l} v_t\Rr)^2\Rr] \leq \sum_{M\in \mcl C_{L,l,\rho,\delta}} \Pr[\M_{L,l} = M] (\B_{L,l} v_t(M))^2 + \Ll(\gamma \log(t_{n+1})\Rr)^d \exp\Ll(-\frac{\rho \vert Q_l\vert \delta^2}{4}\Rr)\Vert u_0\Vert^2_{L^{\infty}} 
\end{align*}
and we finish this step by 
\begin{equation}\label{eq:ApproxDensity}
\begin{split}
\Delta_n \leq & \Er[(u_0)^2] +  (t_{n+1})^{\frac{d+2}{2}}\Ll(\gamma \log(t_{n+1})\Rr)^d \exp\Ll(-\frac{\rho \vert Q_l\vert \delta^2}{4}\Rr)\Vert u_0\Vert^2_{L^{\infty}} \\
& \qquad + \sum_{M\in \mcl C_{L,l,\rho,\delta}} \Pr[\M_{L,l} = M] \int_{t_n}^{t_{n+1}} 2(d+2)t^{\frac{d}{2}} (\B_{L,l} v_t(M))^2 \, \d t.
\end{split}
\end{equation}
We remark that the parameter $\delta > 0$ will be fixed at the end of the proof.

\textit{Step 4: Perturbation estimate.}
It remains to estimate the term $(\B_{L,l} v_t(M))^2$ for the the $\delta$-good configuration. Now we put the expression of $v_t$ in and obtain 
\begin{align*}
(\B_{L,l} v_t(M))^2 &= \Ll(\frac{1}{\vert \mcl Z_K \vert} \sum_{y \in \mcl Z_K }  (\B_{L,l} (u_t - \tau_{y} u_t))(M)\Rr)^2,
\end{align*}
and our aim is to control 
\begin{equation}\label{eq:PerTwo}
\begin{split}
\int_{t_n}^{t_{n+1}} 2(d+2)t^{\frac{d}{2}} \Ll(\frac{1}{\vert \mcl Z_K  \vert} \sum_{y \in \mcl Z_K }  (\B_{L,l} (u_t - \tau_{y} u_t))(M)\Rr)^2 \, \d t .
\end{split}
\end{equation}

To treat \text{\cref{eq:PerTwo}}, we calculate the Radon-Nikodym derivative that 
\begin{align}\label{eq:defgM}
g_M := \frac{\d \Pr[\cdot \vert \M_{L,l} = M]}{\d \Pr} = \frac{1}{\Pr[\M_{L,l} = M]]} \Ind{\M_{L,l} = M]}.
\end{align}
Then we use the reversibility of the semigroup $P_t$ and denote by $g_{M,t} := P_t g_M$
\begin{align*}
\B_{L,l}(u_t - \tau_y u_t)(M) = \Er[g_M  (u_t - \tau_y u_t)] = \Er\Ll[g_{M,t}  (u - \tau_y u)\Rr]. 
\end{align*}
Then we would like to apply the a perturbation estimate \Cref{prop:Perturbation} to control it: let $l_k := l_u + 2k$ then for any $\vert y \vert \leq k$, we have 
\begin{align*}
\Er[g_M  (u_t - \tau_y u_t)] \leq  C(d) (l_k \Vert u \Vert_{L^\infty})^2  \Diri_{Q_{l_k}}(\sqrt{g_M}),
\end{align*}
where $\Diri_{Q_{l_k}}(\sqrt{g_M})$ is a localized Dirichlet form defined in \cref{eq:DirichletLocal}. A heuristic analysis of order is $\Diri_{Q_{l_k}}(\sqrt{g_M}) \simeq O\Ll((l_k)^d\Rr)$ since it is a Dirichlet form on $Q_{l_k}$. If we choose $k = K$ here to cover all the term, the bound will be of order $O(K^d)$, which is big when $K \simeq \sqrt{t} \geq l_u$. Therefore, we apply a coarse-graining argument: let $\overline{[0,y]}_k := \{z_i\}_{0 \leq i \leq n(y)}$ be a lattice path that of scale $k$, ${z_0 = 0}, {z_{n(y)} = y},{\{z_i\}_{1 \leq i < n(y)} \in (k\mathbb{Z})^d}$ so the length of path is the shortest one. (See \Cref{fig:coarse-graining} for illustration.) Then we have 
\begin{align*}
(u - \tau_y u) = \sum_{i=0}^{n(y) - 1} (\tau_{z_i} u - \tau_{z_{i+1}} u) = \sum_{i=0}^{n(y) - 1} \tau_{z_i} (u - \tau_{h_{z_i}}u),
\end{align*}
where $h_{z_i} = z_{i+1} - z_i$ the vector connecting the two and $\vert h_{z_i} \vert \leq k$. This expression with the transport invariant law of Poisson point process, Cauchy-Schwartz inequality implies
\begin{equation}\label{eq:Telescope}
\begin{split}
\Ll(\B_{L,l}(u_t - \tau_y u_t)(M)\Rr)^2 &= \Ll(\sum_{z \in \overline{[0,y]}_k} \Er\Ll[g_{M,t}  \tau_z(u - \tau_{h_z} u)\Rr]\Rr)^2 \\
&= \Ll(\sum_{z \in \overline{[0,y]}_k} \Er\Ll[\Ll(\tau_{-z}g_{M,t}  \Rr)(u - \tau_{h_z} u)\Rr]\Rr)^2 \\
&\leq C(d) n(y) \sum_{z \in \overline{[0,y]}_k} \Ll(\Er\Ll[\Ll(\tau_{-z}g_{M,t}\Rr)  (u - \tau_{h_z} u)\Rr]\Rr)^2.
\end{split}
\end{equation}
This term appears a perturbation estimate, which will be proved in \Cref{prop:Perturbation} that 
\begin{align*}
\Ll(\Er\Ll[\Ll(\tau_{-z}g_{M,t}\Rr)  (u - \tau_{h_z} u)\Rr]\Rr)^2 &\leq C(d) (l_k \Vert u \Vert_{L^\infty})^2  \Diri_{Q_{l_k}}\Ll(\sqrt{\tau_{-z}g_{M,t}}\Rr) \\
& = C(d) (l_k \Vert u \Vert_{L^\infty})^2  \Diri_{\tau_z Q_{l_k}}\Ll(\sqrt{g_{M,t}}\Rr),
\end{align*}
where in the last step we use the transport invariant property of Poisson point process. Now we turn to the choice of the scale $k$. By the heuristic analysis that every $\Diri_{Q_{l_k}}$ contributes order $O((l_k)^d)$ and taking in account $n(y) \leq K/k$ we have in \cref{eq:Telescope}
\begin{align*}
\Ll(\B_{L,l}(u_t - \tau_y u_t)(M)\Rr)^2 \simeq  O\Ll(\Ll(\frac{K}{k}\Rr)^2 (l_k)^{d+2}\Rr) \simeq  O\Ll(\Ll(\frac{K}{k}\Rr)^2 (l_u + 2k)^{d+2}\Rr).
\end{align*}
From this we see that a good scale should be $k = l_u$ so the term above is of order $O(K^2(l_u)^d)$. We put these estimate back to \cref{eq:PerTwo}
\begin{multline}\label{eq:PerTwoa}
\text{\cref{eq:PerTwo}} \leq \Vert u \Vert^2_{L^\infty} \int_{t_n}^{t_{n+1}} 2(d+2)t^{\frac{d}{2}}  K l_u \Ll(\frac{1}{\vert \mcl Z_K \vert}\sum_{y \in \mcl Z_K}   \sum_{z \in \overline{[0,y]}_{l_u}} \Diri_{\tau_z Q_{3 l_u}}\Ll(\sqrt{g_{M,t}}\Rr)\Rr)\, \d t .
\end{multline}

\begin{figure}[h]
\centering
\includegraphics[scale=0.45]{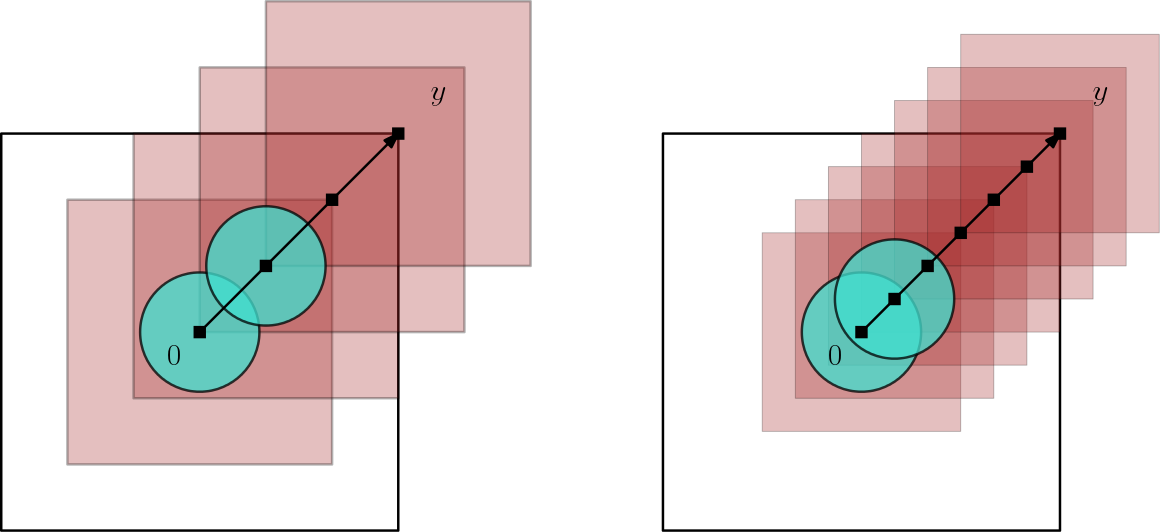}
\caption{The illustration of the coarse-graining argument, where we take a lattice path of scale $k$ to connect $0$ and $y$. The ball in blue is the support of $u$ and the box in red is $Q_{l_k}$. For the one on the left, the scale is $k = l_u$; the one on the right the scale is finer and we see that the coarse-graining is too dense. }\label{fig:coarse-graining}
\end{figure}

\textit{Step 5: Covering argument.} In this step, we calculate the  right hand side of \cref{eq:PerTwoa}, where we notice one essential problem: there are totally about $K^{d+1}/l_u$ terms of Dirichlet form $\Diri_{\tau_z Q_{3 l_u}}\Ll(\sqrt{g_{M,t}}\Rr)$ in the sum $\sum_{y \in \mcl Z_K} \sum_{z \in \overline{[0,y]}_{l_u}} \Diri_{\tau_z Q_{l_u}}\Ll(\sqrt{g_{M,t}}\Rr)$, but the one with $z$ close to $0$ are counted of order $K^d$ times, while the one with $z$ near $\partial \mcl Z_K$ are counted only constant times. To solve this problem, we have to reaverage the sum: by the transport invariant property of Poisson point process, at the beginning of the Step 1, we can write 
\begin{align*}
\Delta_n = \frac{1}{\vert \mcl Z_l \vert}\sum_{x \in \mcl Z_l}\Ll((t_{n+1})^{\frac{d+2}{2}} \Er[(\tau_x v_{t_{n+1}})^2] - (t_n)^{\frac{d+2}{2}} \Er[(\tau_x v_{t_n})^2]\Rr).
\end{align*}
Then all estimates works in Step 1, Step 2 and Step 3 work by replacing ${v_t \mapsto \tau_x v_t}$ and ${u_t \mapsto \tau_x u_t}$. In the Step 4, this operation will change our object term \cref{eq:PerTwo} 
\begin{align*}
\text{\cref{eq:PerTwo}-avg} = \int_{t_n}^{t_{n+1}} 2(d+2)t^{\frac{d}{2}} \Ll(\frac{1}{\vert \mcl Z_l \vert}\sum_{w \in \mcl Z_l}\frac{1}{\vert \mcl Z_K \vert}\sum_{y \in \mcl Z_K} (\B_{L,l} \tau_w(u_t - \tau_y u_t)(M))^2\Rr) \, \d t,
\end{align*}
and the perturbation argument \Cref{prop:Perturbation} reduces the problem as 
\begin{multline}\label{eq:PerTwoaNew}
\text{\cref{eq:PerTwo}-avg} \leq \Vert u \Vert^2_{L^\infty} \int_{t_n}^{t_{n+1}} 2(d+2)t^{\frac{d}{2}}K l_u \\
\qquad \qquad \times \Ll(\frac{1}{\vert \mcl Z_l \vert}\sum_{w \in \mcl Z_l} \frac{1}{\vert \mcl Z_K \vert}\sum_{y \in \mcl Z_K}   \sum_{z \in \overline{[0,y]}_{l_u}} \Diri_{\tau_{w+z} Q_{3 l_u}}\Ll(\sqrt{g_{M,t}}\Rr)\Rr)\, \d t .\\
\end{multline}
Now we can apply the Fubini's lemma
\begin{align*}
&\frac{1}{\vert \mcl Z_l \vert}\sum_{w \in \mcl Z_l} \frac{1}{\vert \mcl Z_K \vert}\sum_{y \in \mcl Z_K}   \sum_{z \in \overline{[0,y]}_{l_u}} \Diri_{\tau_{w+z} Q_{3 l_u}}\Ll(\sqrt{g_{M,t}}\Rr) \\
= & \frac{1}{\vert \mcl Z_l \vert} \frac{1}{\vert \mcl Z_K \vert} \Er\Ll[\int_{\Rd} \Ll(\sum_{w \in \mcl Z_l}\sum_{y \in \mcl Z_K}\sum_{z \in \overline{[0,y]}_{l_u}}\Ind{x \in \tau_{w+z} Q_{3 l_u}} \Rr) \nabla \sqrt{g_{M,t}}(\mu, x) \cdot \nabla \sqrt{g_{M,t}}(\mu, x)\, \d \mu(x)\Rr],
\end{align*}
while we notice that 
\begin{multline*}
\sum_{w \in \mcl Z_l}\sum_{y \in \mcl Z_K}\sum_{z \in \overline{[0,y]}_{l_u}}\Ind{x \in \tau_{w+z} Q_{3 l_u}} = \sum_{y \in \mcl Z_K}\sum_{z \in \overline{[0,y]}_{l_u}}\underbrace{\sum_{w \in \mcl Z_l}\Ind{x-w \in \tau_{z} Q_{3 l_u}}}_{\leq \vert Q_{3 l_u} \vert}\\
\leq  \sum_{y \in \mcl Z_K}\sum_{z \in \overline{[0,y]}_{l_u}} (3 l_u)^d \leq C(d)(l_u)^{d-1} K^{d+1},
\end{multline*}
so we have
\begin{align*}
\frac{1}{\vert \mcl Z_l \vert}\sum_{w \in \mcl Z_l} \frac{1}{\vert \mcl Z_K \vert}\sum_{y \in \mcl Z_K}   \sum_{z \in \overline{[0,y]}_{l_u}} \Diri_{\tau_{w+z} Q_{3 l_u}}\Ll(\sqrt{g_{M,t}}\Rr) \leq  \frac{C(d)(l_u)^{d-1} K}{\vert \mcl Z_l \vert} \Diri(\sqrt{g_{M,t}}).
\end{align*}
We put this estimate to \cref{eq:PerTwoaNew} and use $l = c\sqrt{t_{n+1}},$ 
\begin{equation}\label{eq:PerTwoa2}
\begin{split}
\text{\cref{eq:PerTwo}-avg} &\leq C(d) \Vert u \Vert^2_{L^\infty}K^2 (l_u)^{d} \int_{t_n}^{t_{n+1}} \Ll(\frac{t^{\frac{1}{2}}}{l}\Rr)^d \Diri(\sqrt{g_{M,t}}) \, \d t \\
&\leq  C(d) \Vert u \Vert^2_{L^\infty} K^2 (l_u)^{d}  \int_{t_n}^{t_{n+1}} \Diri(\sqrt{g_{M,t}}) \, \d t.
\end{split}
\end{equation}
We put \cref{eq:PerTwoa2} back to \cref{eq:PerTwo} and \cref{eq:ApproxDensity} and conclude
\begin{equation}
\begin{split}
\Delta_n \leq & \Er[(u_0)^2] +  (t_{n+1})^{\frac{d+2}{2}}\Ll(\gamma \log(t_{n+1})\Rr)^d \exp\Ll(-\frac{\rho \vert Q_l\vert \delta^2}{4}\Rr)\Vert u_0\Vert^2_{L^{\infty}} \\
& \qquad +  C(d) \Vert u \Vert^2_{L^\infty} K^2 (l_u)^{d}\sum_{M\in \mcl C_{L,l,\rho,\delta}} \Pr[\M_{L,l} = M] \int_{t_n}^{t_{n+1}} \Diri(\sqrt{g_{M,t}}) \, \d t.
\end{split}
\end{equation}

\textit{Step 6: Entropy inequality.}
In this step, we analyze the quantity $\int_{t_n}^{t_{n+1}} \Diri(\sqrt{g_{M,t}})\, \d t$. We recall the definition of the entropy inequality: let $H(g_M) = \Er[g_M \log(g_M)]$, then 
\begin{align}
H(g_{M,t}) = H(g_{M}) - 4\int_{0}^t \Er[\sqrt{g_{M,s}}(-\L \sqrt{g_{M,s}})]\, \d s,
\end{align}
we have 
\begin{align*}
\int_{t_n}^{t_{n+1}} \Diri(\sqrt{g_{M,t}})\, \d t \leq \int_{t_n}^{t_{n+1}} \Er[\sqrt{g_{M,t}}(-\L \sqrt{g_{M,t}})]\, \d t \leq H(g_{M,t_{n+1}}) \leq H(g_M).
\end{align*}
For any $M \in \mcl C_{L,l,\rho,\delta}$, one can calculate the bound of the entropy and we prove it in \Cref{lem:Entropy}
\begin{align*}
H(g_M) \leq C(d, \rho)\Ll(\frac{L}{l}\Rr)^d \Ll(\log(l) + l^d\delta^2\Rr).
\end{align*}
This helps us conclude that 
\begin{multline*}
\Delta_n \leq \Er[(u_0)^2] \\
+  \Vert u_0\Vert^2_{L^{\infty}} \Ll(\gamma \log(t_{n+1})\Rr)^d \Ll((t_{n+1})^{\frac{d+2}{2}}\exp\Ll(-\frac{\rho \vert Q_l\vert \delta^2}{4}\Rr) + K^2 (l_u)^{d} \Ll(\log(l) + l^d\delta^2\Rr)\Rr).
\end{multline*}
To make the bound small, we choose a parameter $\delta = c(d,\rho) (\log{t_{n+1}})^{\frac{1}{2}}(t_{n+1})^{-\frac{d}{2}}$, where $c(d,\rho)$ is a positive number large enough to compensate the term $(t_{n+1})^{\frac{d+2}{2}}$ and this proves \cref{eq:Approximation}
\end{proof}

\section{Localization estimate}\label{sec:Localization}
In this part, we prove the key localization estimate: we recall our notation of conditional expectation here that $\A_s f = \E\Ll[f | \mcl F_{Q_s} \Rr]$ for $Q_s$ a closed cube $\Ll[- \frac{s}{2}, \frac{s}{2}\Rr]^d$.
\begin{theorem}\label{thm:localization}
For $u \in L^2\Ll(\mmd(\Rd)\Rr)$ of compact support that $\supp(u) \subset Q_{l_u}$, any ${t \geq \max\Ll\{l_u^2, 16 \Lambda^2 \Rr\}}$, $K \geq \sqrt{t}$, and $u_t$ the function associated to the generator $\L$ at time $t$, then we have the estimate
\begin{equation}
\Er\Ll[(u_t - \A_K u_t)^2\Rr] \leq C(\Lambda)\exp\left(- \frac{K}{\sqrt{t}} \right)\Er\Ll[u^2\Rr].
\end{equation}
\end{theorem}

This is an important inequality which allows us to pay some error to localize the function, and it is introduced  in \cite{janvresse1999relaxation} and also used in \cite{giunti2019heat}. The main idea to prove it is to use a multi-scale functional and analyze its evolution with respect to the time. Let us introduce its continuous version: for any $f \in L^2(\mmd(\Rd))$, $f \mapsto \Ll(\A_s f \Rr)_{s \geq 0}$ is a c\`adl\`ag $L^2$-martingale with respect to $\Ll(\Omega, \Ll(\mcl F_{Q_s}\Rr)_{s \geq 0}, \P\Rr)$. 

Our multi-scale functional for  $f \in H^1_0(\mmd(\Rd))$ is defined as 
\begin{equation}\label{eq:FunctionalMultiscaleContinous}
S_{k,K,\beta}(f) = \alpha_k \Er\Ll[(\A_k f)^2\Rr] + \int_{k}^K \alpha_s \,d \Er\Ll[(\A_s f)^2\Rr] + \alpha_K \Er\Ll[ (f - \A_K f)^2\Rr],
\end{equation}
with $\alpha_s = \exp\Ll(\frac{s}{\beta}\Rr), \beta > 0$. We can apply the integration by part formula for the Lebesgue-Stieltjes integral and obtain 
\begin{equation}\label{eq:FunctionalMultiscaleContinous2}
S_{k,K,\beta}(f) = \alpha_K \Er\Ll[f^2\Rr] -  \int_{k}^K \alpha'_s \Er\Ll[(\A_s f)^2\Rr]\, ds,
\end{equation}
where $\alpha'_s$ is the derivative with respect to $s$. The main idea is to put $u_t$ in \cref{eq:FunctionalMultiscaleContinous2} and then study its derivative $\frac{d}{dt} S_{k,K,\beta}(u_t)$ and use it to prove \Cref{thm:localization}. In this procedure, we will use the Dirichlet form for $\A_s u_t$, but we have to remark that in fact we do not know \`a priori this is a function in $ H^1_0(\mmd(\Rd))$. We will give a counter example to make it more clear in the next section and introduce a regularized version of $\A_s f$ to pass this difficulty.

\subsection{Conditional expectation, spatial martingale and its regularization}
$\Ll(\A_s f \Rr)_{s \geq 0}$ has nice property: we can treat it as a localized function or a martingale. Thus we use the notation 
\begin{align}
\mt^f_s := \A_s f,
\end{align}
which is a more canonical notation in martingale theory. In this subsection, we would like to understand the regularity of the closed martingale $\Ll(\mt^f_s \Rr)_{s \geq 0}$. We will see it is a c\`adl\`ag $L^2$-martingale and the jump happens when there is particles on the boundary $\partial Q_s$. At first, we remark a useful property for Poisson point process.

\begin{lemma}\label{lem:OneParticle}
With probability $1$, for any $0 < s < \infty$, there is at most one particle one the boundary $\partial Q_s$.
\end{lemma}
\begin{proof}
We denote by 
\begin{align*}
\mcl N := \{\mu : \exists 0 < s < \infty, \text{ there exist more than two particles on } \partial Q_s \}.
\end{align*}
Then we choose an increasing sequence $\{s^{\epsilon}_{k}\}_{k \geq 0}$ with $s^{\epsilon}_{0} = 0$, such that
\begin{align*}
\Rd = \bigsqcup_{k=1}^{\infty}C_{s^{\epsilon}_{k}}, 
\qquad C_{s^{\epsilon}_{k}} := Q_{s^{\epsilon}_{k}} \backslash Q_{s^{\epsilon}_{k-1}}, 
\qquad \vert C_{s^{\epsilon}_{k}} \vert = \frac{\epsilon}{k}.
\end{align*}
Then we have that 
\begin{align*}
\Pr \Ll[\mcl N\Rr] &\leq \Pr \Ll[\exists k, \mu(C_{s^{\epsilon}_{k}}) \geq 2\Rr]  \\
&\leq \sum_{k = 1}^{\infty} \Pr \Ll[\mu(C_{s^{\epsilon}_{k}}) \geq 2\Rr]\\
&\leq \sum_{k = 1}^{\infty} \Ll(\rho \vert C_{s^{\epsilon}_{k}}\vert \Rr)^2 \\ 
&\leq (\rho \epsilon)^2.
\end{align*}
We we let $\epsilon$ go down to $0$ and prove that $\Pr \Ll[\mcl N\Rr] = 0$.
\end{proof}

For this reason, in the following, we can do modification of the probability space and always suppose that there is at most one particle on the boundary. This helps us to prove the following regularity property for $\Ll(\mt^f_s \Rr)_{s \geq 0}$.

\begin{lemma}\label{lem:MtJump}
After a modification, for any $f \in C^{\infty}_c(\mmd(\Rd))$ the process $\Ll(\mt^f_s \Rr)_{s \geq 0}$ is a c\`adl\`ag $L^2$-martingale with finite variation, and the discontinuity point occurs for $s$ such that $\mu(\partial Q_s) = 1$. 
\end{lemma}
\begin{proof}
By the classical martingale theory, we know that $\{\mcl{F}_{Q_s}\}_{s \geq 0}$ is a right continuous filtration, thus after a modification the process is c\`adl\`ag. Moreover, from \Cref{lem:OneParticle} we can modify the value to $0$ on a negligible set so that $\mu(\partial Q_s) \leq 1$ for all positive $s$. It remains to prove that if $\mu(\partial Q_s) = 0$, then the process is also left continuous. In this case, there exists a $0 <  \epsilon_0 < s$ such that for any $0 < \epsilon < \epsilon_0$, we have $\mu(Q_{s-\epsilon}) = \mu(Q_s)$. Then 
\begin{align*}
\A_s f(\mu) = \A_s f(\mu \mres Q_s) = \A_s f(\mu \mres Q_{s - \epsilon}).
\end{align*}
We use $\mu \mres Q_{s} = (\mu \mres Q_{s - \epsilon}) + (\mu \mres (Q_s \backslash Q_{s - \epsilon}))$, then 
\begin{align*}
\A_{s-\epsilon} f(\mu) &= \Er \Ll[ \A_s f(\mu \mres Q_{s}) \vert \mcl{F}_{Q_{s-\epsilon}} \Rr] \\
&= \Er \Ll[ \A_s f(\mu \mres Q_{s - \epsilon} + \mu \mres (Q_s \backslash Q_{s - \epsilon})) \vert \mcl{F}_{Q_{s-\epsilon}} \Rr]\\
&= \Pr\Ll[\mu(Q_s \backslash Q_{s - \epsilon}) = 0\Rr]\A_s f(\mu \mres Q_{s - \epsilon}) \\
& \qquad  +  \Pr\Ll[\mu(Q_s \backslash Q_{s - \epsilon}) \geq 1\Rr] \Er\Ll[\A_s f \vert \mcl{F}_{Q_{s-\epsilon}}, \mu(Q_s \backslash Q_{s - \epsilon}) \geq 1 \Rr]\\
&= e^{-\rho\vert Q_s \backslash Q_{s-\epsilon}\vert}  \A_s f(\mu \mres Q_{s - \epsilon}) + \left(1 - e^{-\rho\vert Q_s \backslash Q_{s-\epsilon}\vert}\right)\Er\Ll[\A_s f \vert \mcl{F}_{Q_{s-\epsilon}}, \mu(Q_s \backslash Q_{s - \epsilon}) \geq 1 \Rr].
\end{align*}
If we suppose that $\Vert f \Vert_{L^\infty}$ is finite, then we have $\lim_{\epsilon \searrow 0} \A_{s-\epsilon} f(\mu) = \A_{s} f(\mu)$. Moreover, we have a estimate that 
\begin{align*}
\Ll\vert\A_{s-\epsilon} f(\mu) - \A_{s} f(\mu) \right\vert \leq C \rho \epsilon s^{d-1} \Vert f \Vert_{L^\infty}.
\end{align*}
This implies that the c\`adl\`ag martingale $\Ll(\mt^f_s \Rr)_{s \geq 0}$ is locally Liptchitz for the continuous part, thus it is almost surely of finite variation.
\end{proof}

The following corollaries are simple applications of the result above.

\begin{corollary}\label{cor:Bracket}
For $f \in C^{\infty}_c(\mmd(\Rd))$, we can define a bracket process for $\Ll(\mt^f_s \Rr)_{s \geq 0}$:
we define that  
\begin{align}\label{eq:Bracket}
\Ll[\mt^f\Rr]_s := \sum_{0 < \tau \leq s}\Ll(\Delta \mt^f_\tau \Rr)^2, \qquad \Delta \mt^f_\tau = \mt^f_\tau - \mt^f_{\tau-}, \qquad \tau \text{ is jump point}.
\end{align}
Then $\Ll(\Ll(\mt^f_s\Rr)^2 - \Ll[\mt^f\Rr]_s\Rr)_{s \geq 0}$ is a martingale with respect to $\Ll(\Omega, \Ll(\mcl F_{Q_s}\Rr)_{s \geq 0}, \Pr\Rr)$.
\end{corollary}
\begin{proof}
This is a direct result from jump process; see \cite[Chapter 4e]{jacod2013limit}.
\end{proof}
 
\begin{corollary}\label{cor:LeftLimit}
Let $x \in \supp(\mu)$, and we define a stopping time for $x$
\begin{align}\label{eq:StoppingTime}
\tau(x) := \min\{s \geq 0 \vert x \in Q_s\},
\end{align}
and the normal direction $\n(x)$ and we define 
\begin{align}
\A_{\tau(x)-}f(\mu - \delta_x + \delta_{x-}) := \lim_{\epsilon \searrow 0} \A_{\tau(x)-\epsilon}f(\mu - \delta_x + \delta_{x-\epsilon \n(x)}). 
\end{align}
Then we have almost surely
\begin{align}
\A_{\tau(x)-}f(\mu) = \A_{\tau(x)} f(\mu - \delta_x), \qquad \A_{\tau(x)-}f(\mu - \delta_x + \delta_{x-}) = \A_{\tau(x)}f(\mu).
\end{align}
\end{corollary}
\begin{proof}
The equation $\A_{\tau(x)-}f(\mu) = \A_{\tau(x)} f(\mu - \delta_x)$ is the result of left continuous: from \Cref{lem:OneParticle} we know with probability $1$ there is only $x$ on $\partial Q_{\tau(x)}$ and $\mu - \delta_x$ does not have particle on the boundary so we apply \Cref{lem:MtJump} and obtain this equation.

For the second equation, we have 
\begin{align*}
\A_{\tau(x)}f(\mu) &= \lim_{\epsilon_1 \searrow 0}\A_{\tau(x)}f(\mu - \delta_x + \delta_{x-\epsilon_1 \n(x)}) \\
&= \lim_{\epsilon_1 \searrow 0}\A_{\tau(x)}f(\mu - \delta_x + \delta_{x-\epsilon_1 \n(x)}) \\
&= \lim_{\epsilon_2 \searrow 0}\lim_{\epsilon_1 \searrow \epsilon_2}\A_{\tau(x)-\epsilon_2}f(\mu - \delta_x + \delta_{x-\epsilon_1 \n(x)}) \\
&= \lim_{\epsilon \searrow 0} \A_{\tau(x)-\epsilon}f(\mu - \delta_x + \delta_{x-\epsilon \n(x)}). 
\end{align*}
In the last step, we use the uniformly left continuous for $\A_s f$ and the continuity with respect to $x$.
\end{proof}

One important remark about the conditional expectation is that in fact for $f \in C_c^{\infty}(\mmd(\Rd))$, we may have $\A_L f \notin C_c^{\infty}(\mmd(\Rd))$. The reason is that the conditional expectation creates a small gap at the boundary for the function. 
Here we give an example of the conditional expectation for $\Er[f \vert \mcl F_{B_r}]$, which is easier to state but it shares the same property of $\A_L f$.

\begin{example}\label{ex:CounterExample}
Let $\eta \in C^{\infty}_c(\Rd)$ be a plateau function: 
\begin{align*}
\supp(\eta) \subset B_1, 0 \leq \eta \leq 1, \eta \equiv 1 \text{ in } B_{\frac{1}{2}},  \eta(x) = \eta(\vert x \vert) \text{ decreasing with respect to } \vert x \vert.
\end{align*}
and we define our function 
\begin{align*}
f(\mu) = \left(\int_{\Rd} \eta(x) \, \d\mu(x)\right) \wedge 3 .
\end{align*}
We define the level set $B_r$ such that 
\begin{align*}
B_r := \left\{x \in \Rd \left| \frac{1}{2} \leq \eta(x) \leq 1 \right. \right\}.
\end{align*}
Then, we have $\Er[f \vert \mcl{F}_{B_r}] \notin C^{\infty}_c(\mmd(\Rd))$.
\end{example}
\begin{proof}
Let $\mu_1 = \mu \mres B_r, \mu_2 = \mu \mres (B_1 \backslash B_r) $, then since $\supp(f) \subset B_1$, we have that 
\begin{align*}
\Er[f \vert \mcl{F}_{B_r}] = \left( \mu_1(\eta) + \mu_2(\eta)\right) \wedge 3
\end{align*}
Let us choose a specific configuration to see that $\Er[f \vert \mcl{F}_{B_r}](\mu)$ is not even continuous:
\begin{align*}
\mu_1 = \delta_{x_1} + \delta_{x_2} + \delta_{x_3}, \text{ where } x_1, x_2 \in B_{\frac{1}{2}}, x_3 \in B_r \backslash B_{\frac{1}{2}}.
\end{align*}
Then we can calculate that $2.5 \leq \mu_1(\eta) < 3$ and $2.5 \leq \Er[f \vert \mcl{F}_{B_r}](\mu) < 3$. However, if we take another $\mu_1$ that 
\begin{align*}
\mu_1 = \delta_{x_1} + \delta_{x_2} + \delta_{x_3} + \delta_{x_4}, \text{ where } x_1, x_2 \in B_{\frac{1}{2}}, x_3 \in B_r \backslash B_{\frac{1}{2}}, x_4 \in B_r.
\end{align*}
Then we see that $\mu_1(\eta) > 3$ and we have $\Er[f \vert \mcl{F}_{B_r}](\mu) = 3$. Therefore, once the 4-th particle $x_4$ enters the ball $B_r$, the value of the function will jump to $3$. From this we conclude that $\Er[f \vert \mcl{F}_{B_r}] \notin C^{\infty}_c(\mmd(\Rd))$.
\end{proof}

\bigskip
To make the conditional expectation more regular, we introduce its regularized version: for any $0 < \epsilon < \infty$, we define 
\begin{align}
\A_{s,\epsilon} f := \frac{1}{\epsilon} \int_{0}^{\epsilon} \A_{s + t} f \, dt,
\end{align}
Then we have the following properties.

\begin{proposition}
For any $f \in H^1_0(\mmd (\Rd))$, the function $\A_{s,\epsilon} f \in H^1_0(\mmd (\Rd))$ and $\Ll(\Er\Ll[ (\A_{s,\epsilon} f)^2 \Rr]\Rr)_{s \geq 0}$ a $C^1$ increasing process.
\end{proposition}
\begin{proof}
We calculate the formula for $\Er\Ll[ (\A_{s,\epsilon} f)^2 \Rr]$:
\begin{align*}
\Er\Ll[ (\A_{s,\epsilon} f)^2 \Rr] = \frac{1}{\epsilon^2} \int_{0}^{\epsilon} \int_{0}^{\epsilon}\Er\Ll[ \A_{s+t_1} f \A_{s+t_2} f \Rr] \, \d t_1 \d t_2.
\end{align*}
As we know that $\Er\Ll[ \A_{s+t_1} f \A_{s+t_2} f \Rr] = \Er\Ll[ (\A_{s+ (t_1 \wedge t_2)} f)^2 \Rr] $, we obtain that 
\begin{align}\label{eq:E2Regularized}
\Er\Ll[ (\A_{s,\epsilon} f)^2 \Rr] = \frac{2}{\epsilon^2} \int_{0}^{\epsilon} (\epsilon - t) \Er\Ll[(\A_{s+ t} f)^2 \Rr] \, \d t.
\end{align}
Then we calculate its derivative that for $0 < h < \epsilon$
\begin{align*}
& \lim_{h \searrow 0}\frac{1}{h}\left(\Er\Ll[ (\A_{s+h, \epsilon} f)^2 \Rr] - \Er\Ll[ (\A_{s, \epsilon} f)^2 \Rr]\right) \\
= & \lim_{h \searrow 0} \frac{2}{h\epsilon^2}\Ll( \int_{\epsilon}^{\epsilon + h}  (\epsilon + h - t)\Er\Ll[ (\A_{s+t} f)^2 \Rr] \, \d t - \int_{0}^{h}  (\epsilon - t)\Er\Ll[ (\A_{s+t} f)^2 \Rr] \, \d t  + \int_{h}^{\epsilon} h \Er\Ll[ (\A_{s+t} f)^2 \Rr] \, \d t \Rr)\\
= & \frac{2}{\epsilon^2} \int_{0}^{\epsilon}  \Er\Ll[ (\A_{s+t} f)^2 \Rr] -  \Er\Ll[ (\A_{s} f)^2 \Rr] \, \d t. 
\end{align*}
In the last step, we use the right continuity and this proves that 
\begin{align}\label{eq:E2ds}
\frac{d}{ds}  \Er\Ll[ (\A_{s, \epsilon} f)^2 \Rr] = \frac{2}{\epsilon^2} \int_{0}^{\epsilon}  \Er\Ll[ (\A_{s+t} f)^2 \Rr] -  \Er\Ll[ (\A_{s} f)^2 \Rr] \, \d t.
\end{align}
Then we calculate the partial derivative. We use the formula that 
\begin{align}\label{eq:Derivative}
\e_k \cdot \nabla \A_{s,\epsilon} f(\mu, x) = \lim_{h \to 0}\frac{1}{h}\Ll(  \frac{1}{\epsilon} \int_{0}^{\epsilon} \A_{s + t} f(\mu - \delta_x + \delta_{x + h \e_k}) - \A_{s + t} f(\mu)\, dt\Rr).
\end{align}
We study this derivative case by case.
\begin{enumerate}
\item Case $x \in Q_{s+\epsilon}^c$. In this case, in \cref{eq:Derivative}, for a $h$ small enough, for any $t \in [0, \epsilon]$, neither $x$ nor $x + h \e_k$ is in $Q_{t+s}$, so we have $\A_{s + t} f(\mu - \delta_x + \delta_{x + h \e_k}) = \A_{s + t} f(\mu \mres Q_{s+t})$. This implies that \cref{eq:Derivative} is $0$ in this case.
\item Case $x \in \op{Q_{s}}$. In this case, for a $h$ small enough, for any $t \in [0, \epsilon]$, both $x$ and $x + h \e_k$ is in $Q_{t+s}$, then we have 
\begin{align*}
\e_k \cdot \nabla \A_{s,\epsilon} f(\mu, x) &= \lim_{h \to 0}\frac{1}{h}\Ll(  \frac{1}{\epsilon} \int_{0}^{\epsilon} \A_{s + t} f(\mu - \delta_x + \delta_{x + h \e_k}) - \A_{s + t} f(\mu)\, dt\Rr) \\
&= \frac{1}{\epsilon} \int_{0}^{\epsilon} \lim_{h \to 0}\frac{1}{h} \Ll(\A_{s + t} f(\mu - \delta_x + \delta_{x + h \e_k}) - \A_{s + t} f(\mu)\Rr)\, dt \\
&= \A_{s,\epsilon} \Ll(\e_k \cdot \nabla f(\mu, x)\Rr).
\end{align*}
\item Case $x \in \Ll( Q_{s+\epsilon} \backslash \op{Q_{s}} \Rr)$, $\e_k$ is the normal direction. In this case, we study at first the situation $\n(x)$ and $h \searrow 0$. We divide \cref{eq:Derivative} in three terms:
\begin{align*}
\e_k \cdot \nabla \A_{s,\epsilon} f(\mu, x) &= \mathbf{I} + \mathbf{II} + \mathbf{III} \\
\mathbf{I} &= \frac{1}{\epsilon} \int_{0}^{\epsilon} \Ind{s + t < \tau(x)} \frac{1}{h} \Ll(\A_{s + t} f(\mu - \delta_x + \delta_{x + h \e_k}) - \A_{s + t} f(\mu)\Rr)\, dt \\
\mathbf{II} &=  \frac{1}{\epsilon} \int_{0}^{\epsilon} \Ind{s + t \geq \tau(x) + h} \frac{1}{h} \Ll(\A_{s + t} f(\mu - \delta_x + \delta_{x + h \e_k}) - \A_{s + t} f(\mu)\Rr)\, dt \\ 
\mathbf{III} &= \frac{1}{\epsilon} \int_{0}^{\epsilon} \Ind{\tau(x) \leq s + t < \tau(x) + h} \frac{1}{h} \Ll(\A_{s + t} f(\mu - \delta_x + \delta_{x + h \e_k}) - \A_{s + t} f(\mu)\Rr)\, dt. 
\end{align*}
The term $\mathbf{I}$ and $\mathbf{II}$ are similar as we have discussed above and we have 
\begin{align*}
\lim_{h \searrow 0} \mathbf{I} + \mathbf{II} = \frac{1}{\epsilon} \int_{0}^{\epsilon} \Ind{s+t > \tau(x)} \A_{s+t} \Ll(\e_k \cdot \nabla f(\mu, x)\Rr)\, dt.
\end{align*}
For the term $\mathbf{III}$, since $x + h\e_k \notin Q_{s+t}$, we have $\A_{s + t} f(\mu - \delta_x + \delta_{x + h \e_k}) = \A_{s + t} f(\mu - \delta_x)$. Then, we use the right continuity of $\A_s f$
\begin{align*}
\lim_{h \searrow 0}\mathbf{III} &= \lim_{h \searrow 0}\frac{1}{h \epsilon} \int_{\tau(x)-s}^{\tau(x)-s+h}  \A_{s + t} f(\mu - \delta_x) - \A_{s + t} f(\mu) \, dt \\
&= \frac{1}{\epsilon} \Ll(\A_{\tau(x)} f(\mu - \delta_x) - \A_{\tau(x)} f(\mu)\Rr). 
\end{align*}
We should also remark that is is also the case we do partial derivative from left, in this case we should pay attention on the term $\mathbf{III}$ which is 
\begin{align*}
\lim_{h \searrow 0}\mathbf{III}' &= \lim_{h \searrow 0} \frac{1}{h\epsilon} \int_{0}^{\epsilon} \Ind{\tau(x)-h \leq s + t < \tau(x)} \Ll(\A_{s + t} f(\mu - \delta_x) - \A_{s + t} f(\mu - \delta_x + \delta_{x - h \e_k})\Rr)\, dt \\
&= \frac{1}{\epsilon} \Ll(\A_{\tau(x)-}f(\mu - \delta_x) - \A_{\tau(x)-}f(\mu - \delta_x + \delta_{x-})\Rr).
\end{align*}
In the last step, we use the left continuity of $A_{\tau(x)} f$ when the particle on the boundary is removed. Thanks to \Cref{cor:LeftLimit}, we know this limit coincide with that of $\mathbf{III}$. In conclusion, we could use the notation \cref{eq:Bracket}
\begin{align}
\Delta A_{\tau(x)} f = A_{\tau(x)} f - A_{\tau(x)-} f,
\end{align}
to unify the two. Thus we see it is nothing but the jump of the  c\`adl\`ag martingale.

\item Case $x \in \Ll( Q_{s+\epsilon} \backslash \op{Q_{s}} \Rr)$, $\e_k$ is not the normal direction. This case is simpler than $\e_k$ is normal direction, where we do not have to consider the term $\mathbf{III}$ in the discussion above.
\end{enumerate}
In summary, we obtain the formula that for any $ x \in \supp(\mu)$
\begin{align}
\nabla \A_{s,\epsilon} f(\mu, x) = \left\{ \begin{array}{ll}
         \A_{s,\epsilon} \Ll(\nabla f(\mu, x)\Rr) & x \in \op{Q_{s}};\\
         \frac{1}{\epsilon} \int_{\tau(x)-s}^{\epsilon} \A_{s+t} \Ll(\nabla f(\mu, x)\Rr)\, dt - \frac{\n(x)}{\epsilon}\Delta A_{\tau(x)} f  & x \in \Ll( Q_{s+\epsilon} \backslash \op{Q_{s}} \Rr);\\
         0 & x \in Q_{s+\epsilon}^c.\end{array} \right. 
\end{align}
Finally, we prove that $\A_{s,\epsilon} f \in H^1_0(\mmd (\Rd))$. It is clear that $\A_{s,\epsilon} f \in L^2(\mmd (\Rd))$ by Jensen's inequality for conditional expectation. For its gradient, we have  
\begin{align*}
\Er\Ll[\int_{\Rd} \vert\nabla \A_{s,\epsilon} \vert^2 \, \d \mu \Rr] & \leq  \Er\Ll[\int_{Q_s} \vert \A_{s,\epsilon} \Ll(\nabla f\Rr) \vert^2 \, \d \mu \Rr] + 2\Er\Ll[\int_{ Q_{s+\epsilon} \backslash \op{Q_{s}} } \Ll\vert \frac{1}{\epsilon} \int_{\tau(x)-s}^{\epsilon} \A_{s+t} \Ll(\nabla f \Rr)\, dt \Rr\vert^2 \, \d \mu \Rr]\\
&  \qquad + \frac{2}{\epsilon^2}\Er\Ll[\int_{ Q_{s+\epsilon} \backslash \op{Q_{s}} } \vert\Delta A_{\tau(x)} f \vert^2 \, \d \mu \Rr].
\end{align*}
For the first and second term in the equation above, we use Jensen's inequality for conditional expectation and Cauchy's inequality that 
\begin{multline*}
\Er\Ll[\int_{Q_s} \vert \A_{s,\epsilon} \Ll(\nabla f\Rr) \vert^2 \, \d \mu \Rr] + 2\Er\Ll[\int_{ Q_{s+\epsilon} \backslash \op{Q_{s}} } \Ll\vert \frac{1}{\epsilon} \int_{\tau(x)-s}^{\epsilon} \A_{s+t} \Ll(\nabla f \Rr)\, dt \Rr\vert^2 \, \d \mu \Rr] \\
\leq \Er\Ll[\int_{Q_s} \vert \nabla f \vert^2 \, \d \mu \Rr] + \frac{2}{\epsilon}\Er\Ll[\int_{ Q_{s+\epsilon} \backslash \op{Q_{s}} } \vert \nabla f \vert^2 \, \d \mu \Rr].
\end{multline*}
For the third term, it is in fact the sum of square of the jump part in the martingale $\Ll(\mt^{f}_s\Rr)_{s \geq 0}$, so we use \Cref{cor:Bracket} that 
\begin{multline*}
\Er\Ll[\int_{ Q_{s+\epsilon} \backslash \op{Q_{s}} } \vert\Delta A_{\tau(x)} f \vert^2 \, \d \mu \Rr] = \Er\Ll[\sum_{s \leq \tau \leq s+\epsilon}  \vert \Delta \mt^{f}_{\tau} \vert^2 \Rr] = \Er\Ll[[\mt^{f}]_{s+\epsilon} - [\mt^{f}]_{s} \Rr] \\
= \Er\Ll[\Ll(\mt^{f}_{s+\epsilon}\Rr)^2 - \Ll(\mt^{f}_{s}\Rr)^2 \Rr] =\Er\Ll[\Ll( A_{s+\epsilon} f\Rr)^2 - \Ll( A_{s} f\Rr)^2 \Rr],
\end{multline*}
where in the last step we also use the $L^2$ isometry for martingale. This concludes the desired result $\A_{s,\epsilon} f \in H^1_0(\mmd (\Rd))$. 
\end{proof}

\subsection{Proof of \Cref{thm:localization}}
In this part, we prove \Cref{thm:localization} in three steps.
\begin{proof}
\textit{Step 1: Setting up.}
We propose a regularized multi-scale functional of \cref{eq:FunctionalMultiscaleContinous}
\begin{equation}\label{eq:FunctionalMultiscaleRegulized1}
S_{k,K,\beta, \epsilon}(f) = \alpha_k \Er\Ll[(\A_{k,\epsilon} f)^2\Rr] + \int_{k}^K \alpha_s  \Ll(\frac{d}{ds}\Er\Ll[(\A_{s, \epsilon} f)^2\Rr]\Rr) \, ds + \alpha_K \Er\Ll[ f^2 - \Ll(\A_{K, \epsilon} f\Rr)^2\Rr],
\end{equation}
where we recall that $\alpha_s = \exp\Ll(\frac{s}{\beta}\Rr)$. The advantage is that $\Er\Ll[(\A_{s, \epsilon} f)^2\Rr]$ is $C^1$ for $s$ from \cref{eq:E2ds}, we can treat it as usual Riemann integral and apply integration by part to obtain an equivalent definition
\begin{equation}\label{eq:FunctionalMultiscaleRegulized2}
S_{k,K,\beta, \epsilon}(f) = \alpha_K \Er\Ll[f^2\Rr] -  \int_{k}^K \alpha'_s \Er\Ll[(\A_{s,\epsilon} f)^2\Rr]\, ds.
\end{equation}
Our object is to calculate $\frac{d}{dt}S_{k,K,\beta, \epsilon}(u_t)$, and we pay attention to $\frac{d}{dt}\Er\Ll[(\A_{s,\epsilon} u_t)^2\Rr]$. We use the formula from \cref{eq:E2Regularized}
\begin{align*}
\frac{d}{dt}\Er\Ll[(\A_{s,\epsilon} u_t)^2\Rr] &= \frac{d}{dt}\frac{2}{\epsilon^2} \int_{0}^{\epsilon} (\epsilon - r) \Er\Ll[(\A_{s+ r} u_t)^2 \Rr] \, \d r \\
&= \frac{d}{dt}\frac{2}{\epsilon^2} \int_{0}^{\epsilon} (\epsilon - r) \Er\Ll[(\A_{s+ r} u_t) u_t\Rr] \, \d r. 
\end{align*}
We define that 
\begin{align}\label{eq:Atilde}
\tilde{\A_{s,\epsilon}} f :=  \frac{2}{\epsilon^2} \int_{0}^{\epsilon} (\epsilon - r) \A_{s+ r} f \, \d r,
\end{align}
and it satisfies similar property as $\A_{s,\epsilon} f $. For example, we have also the formula 
\begin{align}\label{eq:AtildeDerivative}
\nabla \tilde{\A_{s,\epsilon}}f(\mu, x) = \left\{ \begin{array}{ll}
         \tilde{\A_{s,\epsilon}} \Ll(\nabla f(\mu, x)\Rr) & x \in \op{Q_{s}};\\
         \frac{2}{\epsilon^2} \Ll(\int_{\tau(x)-s}^{\epsilon}(\epsilon - r)\A_{s+r} \Ll(\nabla f(\mu, x)\Rr)\, dr - (s+ \epsilon  - \tau(x))\Delta A_{\tau(x)} f \n(x)\Rr) & x \in \Ll( Q_{s+\epsilon} \backslash \op{Q_{s}} \Rr);\\
         0 & x \in Q_{s+\epsilon}^c.\end{array} \right. 
\end{align}
then we have 
\begin{align}\label{eq:dtSemigroupConditional}
\frac{d}{dt}\Er\Ll[(\A_{s,\epsilon} u_t)^2\Rr] = \frac{d}{dt}\Er\Ll[\Ll(\tilde{\A_{s,\epsilon}}u_t\Rr) u_t\Rr] = \Er\Ll[\Ll(\frac{d}{dt}\tilde{\A_{s,\epsilon}}u_t\Rr) u_t\Rr] + \Er\Ll[\tilde{\A_{s,\epsilon}}u_t (\L u_t)\Rr].
\end{align}
We study at first the semi-group. For a function $g \in H^1_0(\mmd(\Rd))$, we recall the definition that 
\begin{align*}
g_t(\mu) = P_t g(\mu) := \Er\Ll[g(\mu_t) | \fil_0\Rr].
\end{align*}
We also know its semi-group that 
\begin{align*}
\frac{d}{dt} P_t g(\mu) = \L P_t g(\mu) \Rightarrow \partial_t g_t(\mu) = \L g_t(\mu).
\end{align*}
Now in our question we propose that $g = \tilde{\A_{s,\epsilon}} u_0$, then we have 
\begin{align*}
g_t(\mu) &= P_t \Ll(\frac{2}{\epsilon^2} \int_{0}^{\epsilon} (\epsilon-r)\Er[u(\mu) | \mcl{F}_{Q_{s+r}}] \, \d r \Rr) \\
&= \Er\Ll[\Ll(\frac{2}{\epsilon^2} \int_{0}^{\epsilon} (\epsilon-r)\Er[u(\mu_t) | \mcl{F}_{Q_{s+r}}] \, \d r \Rr)\Ll \vert \fil_0 \right.\Rr] \\
&= \frac{2}{\epsilon^2} \int_{0}^{\epsilon} (\epsilon-r)\Er\Ll[ \Er\Ll[ u(\mu_t)\Ll \vert \fil_0 \right.\Rr]  | \mcl{F}_{Q_{s+r}}\Rr] \, \d r \\
&= \tilde{\A_{s,\epsilon}} u_t(\mu). 
\end{align*}
Therefore, we have $\frac{d}{dt} \tilde{\A_{s,\epsilon}} u_t(\mu)= \L \tilde{\A_{s,\epsilon}} u_t(\mu)$ and put it back to \cref{eq:dtSemigroupConditional} and use reversibility to obtain that 
\begin{align*}
\frac{d}{dt}\Er\Ll[(\A_{s,\epsilon} u_t)^2\Rr] = 2\Er\Ll[\tilde{\A_{s,\epsilon}} u_t (\L u_t)\Rr].
\end{align*}
We conclude that
\begin{align}\label{eq:ThmLocPfSetup}
\frac{d}{dt} S_{k,K,\beta, \epsilon}(u_t) = 2\alpha_K \Er\Ll[u_t (\L u_t)\Rr] + \int_{k}^K 2\alpha'_s \Er\Ll[\tilde{\A_{s,\epsilon}}u_t (-\L u_t)\Rr]\, ds.
\end{align}

\textit{Step 2: Estimate of a localized Dirichlet energy.}
In this step, we will give an estimate for the term $\Er\Ll[\tilde{\A_{s,\epsilon}}u_t (-\L u_t)\Rr]$ appeared in \cref{eq:ThmLocPfSetup}. We will establish the following lemma.

\begin{lemma}
For any $f \in H^1_0(\mmd(\Rd))$, we define that 
\begin{equation}\label{eq:I}
I^f_s := \Er\Ll[ \int_{Q_s} \nabla f \cdot \a  \nabla f   \, \d\mu \Rr],
\end{equation}
then for $\tilde{\A_{s,\epsilon}} f$ introduced in \cref{eq:Atilde}, for any $s, \theta, \epsilon \in (0,\infty)$, we have 
\begin{equation}\label{eq:Perturbation}
\Er\Ll[\tilde{\A_{s,\epsilon}} f (-\L f)\Rr]  \leq  I^f_{s-1}  + \Lambda\Ll(I^f_{s} - I^f_{s-1}\Rr) + \Lambda\Ll(\frac{\theta}{\epsilon} + 1\Rr) \Ll(I^f_{s+\epsilon} - I^f_{s}\Rr) + \frac{\Lambda}{2\theta} \frac{d}{ds} \Er\Ll[\Ll(\A_{s,\epsilon} f\Rr)^2 \Rr].
\end{equation}
\end{lemma}
\begin{proof}
From \cref{eq:AtildeDerivative}, we can decompose the quantity $\Er\Ll[\tilde{\A_{s,\epsilon}} f (-\L f)\Rr]$ into three terms
\begin{equation}\label{eq:PerThreeTerms}
\begin{split}
\Er\Ll[\tilde{\A_{s,\epsilon}} f (-\L f)\Rr] & =  \underbrace{\Er\Ll[ \int_{Q_{s-1}} \nabla (\tilde{\A_{s,\epsilon}} f) \cdot \a \nabla f \, \d\mu\Rr]}_{\text{\cref{eq:PerThreeTerms}-a}}  +  \underbrace{\Er\Ll[ \int_{Q_s \backslash Q_{s-1}}  \nabla (\tilde{\A_{s,\epsilon}} f) \cdot \a \nabla f   \, \d\mu\Rr]}_{\text{\cref{eq:PerThreeTerms}-b}} \\
& \qquad + \underbrace{\Er\Ll[ \int_{Q_{s+\epsilon} \backslash Q_{s}}  \nabla (\tilde{\A_{s,\epsilon}} f) \cdot \a \nabla f   \, \d\mu\Rr]}_{\text{\cref{eq:PerThreeTerms}-c}}.
\end{split}
\end{equation}

For the first term \cref{eq:PerThreeTerms}-a, since $x \in Q_{s-1}$, then the coefficient is $\mcl F_{Q_s}$ measurable. We use the formula \cref{eq:AtildeDerivative}, \cref{eq:Atilde} and apply Jensen's inequality for conditional expectation 
\begin{align*}
\text{\cref{eq:PerThreeTerms}-a} &= \frac{2}{\epsilon^2}\Er\Ll[ \int_{Q_{s-1}}   \int_{0}^{\epsilon} (\epsilon - r)  \A_{s+ r} (\nabla f)  \cdot \a \nabla f \, \d r  \, \d\mu\Rr] \\
&= \frac{2}{\epsilon^2}\Er\Ll[ \int_{Q_{s-1}}   \int_{0}^{\epsilon} (\epsilon - r) \Er\Ll[\A_{s+ r} (\nabla f) \cdot \a \A_{s+ r} (\nabla f) \, \vert \mcl{F}_{Q_{s+r}}\Rr] \, \d r  \, \d\mu\Rr] \\
&\leq \frac{2}{\epsilon^2}\Er\Ll[ \int_{Q_{s-1}}   \int_{0}^{\epsilon} (\epsilon - r) \Er\Ll[\nabla f \cdot \a  \nabla f \, \vert \mcl{F}_{Q_{s+r}}\Rr] \, \d r  \, \d\mu\Rr]\\
&= \Er\Ll[ \int_{Q_{s-1}}  \nabla f \cdot \a  \nabla f   \, \d\mu\Rr]
\end{align*}

For the second term \cref{eq:PerThreeTerms}-b, it is similar but $\a$ is no longer $\mcl{F}_{Q_s}$ measurable. We use at first  Young's inequality
\begin{align*}
\text{\cref{eq:PerThreeTerms}-b} &\leq \frac{2}{\epsilon^2}\Er\Ll[ \int_{Q_s \backslash Q_{s-1}}   \int_{0}^{\epsilon} (\epsilon - r)  \A_{s+ r} (\nabla f)  \cdot \a \nabla f \, \d r  \, \d\mu\Rr] \\
&\leq \frac{\Lambda}{\epsilon^2}\Er\Ll[ \int_{Q_s \backslash Q_{s-1}}   \int_{0}^{\epsilon} (\epsilon - r)  \Ll( \Ll\vert\A_{s+ r} (\nabla f)\Rr\vert^2 +  \Ll\vert\nabla f \Rr\vert^2 \Rr) \, \d r  \, \d\mu\Rr].
\end{align*}
Then for the part with conditional expectation, we use the uniform bound $1 \leq \a \leq \Lambda$ that 
\begin{align*}
\frac{\Lambda}{\epsilon^2}\Er\Ll[ \int_{Q_s \backslash Q_{s-1}}   \int_{0}^{\epsilon} (\epsilon - r) \Ll\vert\A_{s+ r} (\nabla f)\Rr\vert^2  \, \d r  \, \d\mu\Rr] & \leq \frac{\Lambda}{2}\Er\Ll[ \int_{Q_s \backslash Q_{s-1}}  \vert \nabla f \vert^2  \d\mu\Rr]  \\
& \leq \frac{\Lambda}{2}\Er\Ll[ \int_{Q_s \backslash Q_{s-1}}  \nabla f \cdot \a  \nabla f  \d\mu\Rr]. 
\end{align*}
This concludes that $\text{\cref{eq:PerThreeTerms}-b} \leq \Lambda \Er\Ll[ \int_{Q_s \backslash Q_{s-1}}  \nabla f \cdot \a  \nabla f  \d\mu\Rr]$.

For the third term \cref{eq:PerThreeTerms}-c, we use \cref{eq:AtildeDerivative} and obtain 
\begin{align*}
\text{\cref{eq:PerThreeTerms}-c} &\leq \text{\cref{eq:PerThreeTerms}-c1} + \text{\cref{eq:PerThreeTerms}-c2} \\
\text{\cref{eq:PerThreeTerms}-c1} &= \frac{2}{\epsilon^2}\Ll\vert\Er\Ll[ \int_{Q_{s+\epsilon} \backslash \op{Q_{s}}}   \int_{\tau(x)-s}^{\epsilon} (\epsilon - r) \A_{s+ r} (\nabla f)  \cdot \a \nabla f \, \d r  \, \d\mu\Rr] \Rr\vert\\
\text{\cref{eq:PerThreeTerms}-c2} &= \frac{2}{\epsilon^2}\Ll\vert\Er\Ll[\int_{Q_{s+\epsilon} \backslash \op{Q_{s}}}  (\epsilon - \tau(x))\Delta A_{\tau(x)} f \n(x) \cdot \a \nabla f \, \d\mu\Rr] \Rr\vert.
\end{align*}
The part of \cref{eq:PerThreeTerms}-c1 is similar as that of \cref{eq:PerThreeTerms}-b and we have that 
\begin{align*}
\text{\cref{eq:PerThreeTerms}-c1} \leq \Lambda \Er\Ll[ \int_{Q_{s+\epsilon} \backslash \op{Q_{s}}}  \nabla f \cdot \a  \nabla f  \d\mu\Rr].
\end{align*}
We study the part \cref{eq:PerThreeTerms}-c2 with Young's inequality
\begin{align*}
&\frac{2}{\epsilon^2}\Ll\vert\Er\Ll[\int_{Q_{s+\epsilon} \backslash \op{Q_{s}}}  (\epsilon - \tau(x))\Delta A_{\tau(x)} f \n(x) \cdot \a \nabla f \, \d\mu \Rr] \Rr\vert \\
\leq & \frac{\Lambda}{\theta\epsilon^2}\Er\Ll[\int_{ Q_{s+\epsilon} \backslash \op{Q_{s}}}   (s+ \epsilon - \tau(x))\Delta \vert A_{\tau(x)} f \vert^2 \, \d\mu\Rr] + \frac{\theta\Lambda}{\epsilon^2}\Er\Ll[\int_{Q_{s+\epsilon} \backslash \op{Q_{s}}}   (s + \epsilon - \tau(x)) \vert \nabla f \vert^2 \, \d\mu \Rr] \\
\leq & \frac{\Lambda}{\theta\epsilon^2}\Er\Ll[\int_{Q_{s+\epsilon} \backslash \op{Q_{s}}}   (s+ \epsilon - \tau(x))\Delta \vert A_{\tau(x)} f \vert^2 \, \d\mu\Rr] + \frac{\theta\Lambda}{\epsilon}\Er\Ll[\int_{ Q_{s+\epsilon} \backslash \op{Q_{s}} } \nabla f \cdot \a  \nabla f \, \d\mu \Rr]. \\
\end{align*}
The first part is in fact the bracket process defined in \Cref{cor:Bracket}
\begin{align*}
\frac{\Lambda}{\theta\epsilon^2}\Er\Ll[\int_{Q_{s+\epsilon} \backslash \op{Q_{s}} }   (s+ \epsilon - \tau(x))\Delta \vert A_{\tau(x)} f \vert^2 \, \d\mu\Rr] =  \frac{\Lambda}{\theta\epsilon^2}\Er\Ll[\sum_{s \leq \tau \leq s+\epsilon}  (s+ \epsilon - \tau) \vert \Delta \mt^{f}_{\tau} \vert^2 \Rr].
\end{align*}
Then we develop it with Fubini theorem and the $L^2$-isometry that ${\Er\Ll[\Ll[\mt^f \Rr]_{s}\Rr] = \Er\Ll[(\mt^f_s)^2\Rr] = \Er\Ll[(\A_s f)^2\Rr]}$
\begin{align*}
\frac{\Lambda}{\theta\epsilon^2}\Er\Ll[\sum_{s \leq \tau \leq s+\epsilon}  (\epsilon - \tau) \vert \Delta \mt^{f}_{\tau} \vert^2 \Rr] &= \frac{\Lambda}{\theta\epsilon^2}\Er\Ll[\sum_{s \leq \tau \leq s+\epsilon}  \int_{s}^{s+\epsilon} \Ind{\tau \leq r \leq s+ \epsilon} \, \d r \vert \Delta \mt^{f}_{\tau} \vert^2 \Rr]\\
&= \frac{\Lambda}{\theta\epsilon^2}\Er\Ll[\int_{s}^{s+\epsilon} \sum_{s \leq \tau \leq r}   \vert \Delta \mt^{f}_{\tau} \vert^2 \, \d r\Rr]\\
&= \frac{\Lambda}{\theta\epsilon^2}\Er\Ll[\int_{s}^{s+\epsilon} \Ll[\mt^f \Rr]_{r} - \Ll[\mt^f \Rr]_{s} \, \d r\Rr] \\
&= \frac{\Lambda}{\theta\epsilon^2}\int_{0}^{\epsilon} \Er\Ll[(\A_{s+r} f)^2\Rr] - \Er\Ll[(\A_s f)^2\Rr] \, \d r \\
&= \frac{\Lambda}{2\theta} \frac{d}{ds} \Er\Ll[(\A_{s, \epsilon} f)^2\Rr].
\end{align*}
In the last step, we use the identity \cref{eq:E2ds}. This concludes that 
\begin{align*}
\text{\cref{eq:PerThreeTerms}-c} \leq \Ll(\frac{\theta\Lambda}{\epsilon} + \Lambda\Rr)\Er\Ll[\int_{ Q_{s+\epsilon} \backslash \op{Q_{s}} } \nabla f \cdot \a  \nabla f \, \d\mu \Rr] + \frac{\Lambda}{2\theta} \frac{d}{ds} \Er\Ll[(\A_{s, \epsilon} f)^2\Rr],
\end{align*}
and we combine all the estimate for the three terms \cref{eq:PerThreeTerms}-a, \cref{eq:PerThreeTerms}-b, \cref{eq:PerThreeTerms}-c to obtain the desired result in \cref{eq:Perturbation}.
\end{proof}

\textit{Step 3: End of the proof.} We take $k = \sqrt{t}, K > k$ and and put the estimate \cref{eq:Perturbation} into \cref{eq:ThmLocPfSetup} with $\theta, \epsilon, \beta> 0$ to be fixed, 
\begin{align*}
&\frac{d}{dt} S_{k,K,\beta, \epsilon}(u_t) \\
=& 2\alpha_K \Er\Ll[u_t (\L u_t)\Rr] + \int_{k}^K 2\alpha'_s \Er\Ll[\tilde{\A_{s,\epsilon}}u_t (-\L u_t)\Rr]\, ds \\
\leq & -2\alpha_K I^{u_t}_{\infty} +  \int_{k}^{K} 2\alpha'_s \Ll\{I^{u_t}_{s-1}  + \Lambda\Ll(I^{u_t}_{s} - I^{u_t}_{s-1}\Rr) + \Lambda\Ll(\frac{\theta}{\epsilon} + 1\Rr) \Ll(I^{u_t}_{s+\epsilon} - I^{u_t}_{s}\Rr) + \frac{\Lambda}{2\theta} \frac{d}{ds} \Er\Ll[\Ll(\A_{s,\epsilon} u_t\Rr)^2 \Rr]\Rr\} \, \d s.
\end{align*}
We recall that $\alpha'_s = \frac{\alpha_s}{\beta}$, then we do some calculus and obtain that
\begin{align*}
\frac{d}{dt} S_{k,K,\beta, \epsilon}(u_t) \leq &\int_{k-1}^{K+\epsilon} \Ll(-2 \alpha_{K \wedge (s+1)} + 2\Lambda(\alpha_{s+1} - \alpha_s) + 2\Lambda \Ll(\frac{\theta}{\epsilon} + 1\Rr)(\alpha_s - \alpha_{s-\epsilon})\Rr)\, d I^{u_t}_{s} \\
& \qquad + \int_{0}^{k-1} -2 \alpha_k \, d I^{u_t}_{s} + \int_{K+\epsilon}^{\infty} -2 \alpha_K \, d I^{u_t}_{s} + \frac{\Lambda}{\beta\theta} \int_{k}^K \alpha_s \Ll(\frac{d}{ds} \Er\Ll[\Ll(\A_{s,\epsilon} u_t\Rr)^2 \Rr]\Rr) \, \d s. 
\end{align*}
We see that the term $2\Lambda(\alpha_{s+1} - \alpha_s) \simeq \frac{2\Lambda}{\beta}\alpha_s$ and $2\Lambda \Ll(\frac{\theta}{\epsilon} + 1\Rr)(\alpha_s - \alpha_{s-\epsilon}) \simeq 2\Lambda\Ll(\frac{\theta}{\beta} + \frac{\epsilon}{\beta}\Rr)\alpha_s$.
One can choose the parameters $\theta = \frac{\beta}{2\Lambda}$, $\epsilon = \frac{1}{2}$, then for $\beta > 4\Lambda$, the part of integration with respect to $I^{u_t}_{s}$ is negative. We use the definition \cref{eq:FunctionalMultiscaleRegulized1} and obtain that 
\begin{align*}
\frac{d}{dt}S_{k,K,\beta, \epsilon}(u_t) \leq \frac{\Lambda}{\beta\theta} \int_{k}^K \alpha_s \Ll(\frac{d}{ds} \Er\Ll[\Ll(\A_{s,\epsilon} u_t\Rr)^2 \Rr]\Rr) \, \d s \leq \frac{2\Lambda^2}{\beta^2} S_{k,K,\beta, \epsilon}(u_t),
\end{align*}
which implies that for $k = \sqrt{t} \geq l_u$, ($l_u$ the diameter of support of $u_0$ in \Cref{thm:localization})
\begin{align*}
\alpha_K \Er\Ll[(u_t)^2 - (\A_{K,\epsilon} u_t)^2\Rr] \leq S_{k,K,\beta,\epsilon}(u_t) \leq \exp\Ll(\frac{2\Lambda^2 t}{\beta^2}\Rr)S_{k,K,\beta,\epsilon}(u_0) = \exp\Ll(\frac{2\Lambda^2 t}{\beta^2}\Rr)\alpha_k \Er\Ll[(u_0)^2\Rr].
\end{align*} 
Finally we remark that 
\begin{align*}
\Er\Ll[(u_t - \A_{K+\epsilon} u_t)^2\Rr] = \Er\Ll[(u_t)^2 - (\A_{K+\epsilon} u_t)^2\Rr] \leq \Er\Ll[(u_t)^2 - (\A_{K,\epsilon} u_t)^2\Rr],
\end{align*}
and choose $\beta = \sqrt{t}$ and it gives us the desired result, after shrinking a little the value of $K$.
\end{proof}

\section{Spectral inequality, perturbation and perturbation}\label{sec:Toolbox}
In this section, we collect several estimates used in the proof of the main result. They can also be read for independent interests.

\subsection{Spectral inequality}
The spectral inequality is an important topic in probability theory and Markov process, and it has its counterpart in analysis known as 
Poincar\'e's inequality. 

Let $L > l >0$ and $q = L/l \in \mathbb{N}$, and denote by $\{Q_l^i\}_{1 \leq i\leq q}$ the partition of $Q_L$ by the small cube by scale $l$. Let $\M_{L,l} = (\M_1, \M_2 \cdots \M_q),$ be a random vector that $\M_i = \mu\Ll(Q^i_l\Rr)$, and we define ${\B_{L,l} f := \Er\Ll[f \vert \M_{L,l} \Rr]}$, then we have the following estimate.
\begin{proposition}[Spectral inequality]\label{prop:AB}
There exists a finite positive number $R_0(d)$, such that for any $0 < l < L < \infty$, $L/l \in \mathbb{N}$,  we have an estimate for any $f \in H^1(\mmd(\Rd))$, 
\begin{align}\label{eq:AB}
\Er\Ll[(\A_L f - \B_{L,l} f)\Rr] \leq R_0 l^2 \Er\Ll[\int_{Q_L} \vert\nabla f \vert^2 \, \d \mu\Rr].
\end{align}
\end{proposition}
\begin{proof}
We prove at first a simple corollary from Efron-Stein inequality \cite[Chapter 3]{boucheron2013concentration}: let $f_n \in C^1({\R^{d \times n}})$ and $X = (X_1, X_2 \cdots X_n)$, where $(X_i)_{1 \leq i \leq n}$ a family independent $\Rd$-valued random variables following uniform law in $Q_l$, then Efron-Stein inequality states
\begin{align}\label{eq:EfronStein}
\var\Ll[ f_n(X)\Rr] \leq \frac{1}{2} \sum_{i=1}^n \E\Ll[ \Ll(f_n(X) - f_n(X^{i})\Rr)^2\Rr],
\end{align}
where $f_n(X^{i}) :=  \E\Ll[ f_n(X) \vert X_1 \cdots X_{i-1}, X_{i+1}, \cdots X_n\Rr]$. From this, we calculate the expectation with respect to $X_i$ for $\Ll(f_n(X) - f_n(X^{i})\Rr)^2$, and apply the standard Poincar\'e's inequality for $X_i$
\begin{align*}
\E_{X_i}\Ll[ \Ll(f_n(X) - f(X^{i})\Rr)^2\Rr] &= \fint_{Q_l} \Ll(f_n(x_1, x_2, \cdots x_n ) -  \fint_{Q_l} f_n(x_1, x_2, \cdots x_n ) \, \d x_i\Rr)^2 \, \d x_i \\
& \leq C(d) l^2 \fint_{Q_l} \vert \nabla_{x_i} f_n \vert^2 (x_1, x_2, \cdots x_n )  \, \d x_i,\\
\Longrightarrow \E \Ll[ \Ll(f_n(X) - f(X^{i})\Rr)^2\Rr] &\leq C(d) l^2  \E[\vert \nabla_{x_i} f_n(X)\vert^2].
\end{align*}
We combine the sum of all the term and obtain 
\begin{align}\label{eq:SpectralLemma}
\var\Ll[ f_n(X)\Rr] \leq C(d)l^2 \sum_{i=1}^n \E[\vert \nabla_{x_i} f_n(X)\vert^2].
\end{align}

We then apply \cref{eq:SpectralLemma} in \cref{eq:AB}.
\begin{align*}
\Er\Ll[(\A_L f - \B_{L,l} f)^2\Rr] = \sum_{M \in \mathbb{N}^{q}} \Pr[\M_{L,l} = M] \Er\Ll[(\A_L f - \B_{L,l} f)^2 \vert \M_{L,l} = M \Rr].
\end{align*}
Conditioned $\{\M_{L,l} = M\}$, we know that the expectation of $\A_L f$ is $\B_{L,l} f(M)$ and all the particles are distributed uniformly in its small cubes of size $l$, thus we can apply \cref{eq:SpectralLemma} that 
\begin{align*}
\Er\Ll[(\A_L f - \B_{L,l} f)^2 \vert \M_{L,l} = M \Rr] &= \var{\rho}\Ll[ \A_L f \vert \M_{L,l}= M \Rr] \\
&\leq C(d)l^2 \Er\Ll[\int_{Q_L}\vert \nabla \A_L f \vert^2 \, \d \mu \Ll\vert \M_{L,l} = M \Rr.\Rr] \\
&\leq C(d)l^2 \Er\Ll[\int_{Q_L}\vert  \A_L \nabla f \vert^2 \, \d \mu \Ll\vert \M_{L,l} = M \Rr.\Rr].
\end{align*}
Then we do the sum and concludes \cref{eq:AB}.
\end{proof}

\subsection{Perturbation}
A similar version of the following lemma appears in \cite{janvresse1999relaxation}, where the authors give some sketch and here we prove it in our model with some more details. We define a localized Dirichlet form for Borel set $U \subset \Rd$ that 
\begin{align}\label{eq:DirichletLocal}
\Diri_U(f,g) = \Er[g(-\Delta_{U} f)] := \Er\Ll[\int_U \nabla g(\mu, d) \cdot \nabla f(\mu, x) \, \d \mu(x)\Rr],
\end{align}
and we use ${\Diri_U(f) := \Diri_U(f,f)}$ and ${\Diri(f) := \Diri_{\Rd}(f)}$ for short.
\begin{proposition}[Perturbation]\label{prop:Perturbation}
Let $u \in C^{\infty}_c(\mmd(\Rd))$ and $l_k := l_u + 2k$ be the minimal scale such that for any $\vert h\vert \leq k, \supp(\tau_h u) \subset Q_{l_k}$, then for any $g$ such that ${\Er[g] = 1}, {\sqrt{g} \in H^1_0(\mmd(\Rd))}$, we have 
\begin{align}
\Ll(\Er[g(u - \tau_h u)]\Rr)^2 \leq C(d) (l_k \Vert u \Vert_{L^\infty})^2  \Diri_{Q_{l_k}}(\sqrt{g}).
\end{align}
\end{proposition}
\begin{proof}
The proof of this proposition relies on the following lemma:
\begin{lemma}[Lemma 4.2 of \cite{janvresse1999relaxation}]\label{lem:4.2}
Let $(\Omega, \P, \mcl F)$ be a probability space and let ${\langle f, g \rangle = \int_{\Omega} fg \, \d \P}$ denote the standard inner product on $L^2(\Omega, \P, \mcl F)$. Let $A$ be a non-negative definite symmetric operator on $L^2(\Omega, \P, \mcl F)$, which has $0$ as a simple eigenvalue with corresponding eigenfunction the constant function $1$, and second eigenvalue $\delta > 0$ (the spectral gap). Let $V$ be a function of means zero, $\langle 1, V\rangle = 0$ and assume that $V$ is essential bounded. Denote by $\lambda_{\epsilon}$ the principal eigenvalue of $-A + \epsilon V$ given by the variational formula 
\begin{align}\label{eq:4.2.1}
\lambda_{\epsilon} = \sup_{\Vert f \Vert_{L^2} = 1} \langle f, (-A + \epsilon V)f\rangle.
\end{align}
Then for $0 < \epsilon < \delta(2 \Vert V \Vert_{L^{\infty}})^{-1}$,
\begin{align}\label{eq:4.2.2}
0 \leq \lambda_{\epsilon} \leq \frac{\epsilon^2 \langle V, A^{-1}V\rangle}{1 - 2\Vert V \Vert_{L^{\infty}} \epsilon \delta^{-1}}.
\end{align}
\end{lemma}

In our context, we should look for a good frame for this lemma. Since for any ${\vert h \vert \leq k}, {(u - \tau_h u) \in \mcl F_{Q_{l_k}}}$, we have 
\begin{equation}\label{eq:PerSetUp}
\begin{split}
\Er[g(u - \tau_h u)] &= \Er[(\A_{Q_{l_k}} g)(u - \tau_h u)]\\
&= \sum_{n=0}^{\infty}\Pr[\mu(Q_{l_k}) = n]\Er[(\A_{Q_{l_k}} g)(u - \tau_h u) | \mu(Q_{l_k}) = n].
\end{split}
\end{equation}
Then, we focus on the estimate of $\Er[(\A_{Q_{l_k}} g)(u - \tau_h u) | \mu(Q_{l_k}) = n]$: to shorten the notation, we use $\P_{\rho, n}$ for the probability $\Pr[\cdot \vert  \mu(Q_{l_k}) = n]$ and $\E_{\rho, n}$ for its associated expectation. Then we apply  \Cref{lem:4.2} on the probability space $(\Omega, \mcl F_{Q_{l_k}}, \P_{\rho, n})$, where we set $V = u - \tau_h u$ and the symmetric non-negative operator $A$ is $(-\Delta_{Q_{l_k}})$ defined for any $f \in H^1(\mmd(Q_{l_k}))$
\begin{align*}
\E_{\rho, n}[f (-\Delta_{Q_{l_k}} f)] := \E_{\rho, n}\Ll[\int_{Q_{l_k}} \vert \nabla f\vert^2 \, \d \mu  \Rr].
\end{align*}
We should check that this setting satisfies the condition of \Cref{lem:4.2}: 
\begin{itemize}
\item Spectral gap for $A = -\Delta_{Q_{l_k}}$: by \cref{eq:EfronStein} we have the spectral gap $\delta = (l_k)^{-2}$ for any function $f \in H^1(\mmd(Q_{l_k}))$ with ${\E_{\rho, n}[f] = 0}$
\begin{align*}
\E_{\rho, n}[f^2] \leq (l_k)^2\E_{\rho, n}[f (-\Delta_{Q_{l_k}} f)].
\end{align*}
\item Mean zero for $V = u - \tau_h u$: under the probability $\Pr$ this is clear by the transport invariant property of Poisson point process, while under $\P_{\rho, n}$ this requires some calculus. By the definition of $l_k$, we know that ${\supp(u) \subset Q_{l_u}}$, thus we denote by the projection ${u(\mu) = \tilde{u}_m(x_1, x_2, \cdots x_m)}$ under the case $\mu \mres Q_{l_u} = \sum_{i=1}^m \delta_{x_i}$. Then we have 
\begin{align*}
\E_{\rho, n}[u] &= \sum_{m=0}^n \P_{\rho, n}[\mu(Q_{l_u})=m]\E_{\rho, n}[u \vert \mu(Q_{l_u})=m] \\
&= \sum_{m=0}^n {n \choose m} \Ll(\frac{\vert Q_{l_u} \vert}{\vert Q_{l_k}\vert}\Rr)^m \Ll(1 - \frac{\vert Q_{l_u} \vert}{\vert Q_{l_k}\vert}\Rr)^{n-m} \fint_{(Q_{l_u})^m} \tilde{u}_m(x_1, \cdots x_m) \, \d x_1 \cdots \d x_m,
\end{align*}
because under $\P_{\rho, n}$, the number of particles in $Q_{l_u}$ follows the law $\text{Bin}\Ll(n, \frac{\vert Q_{l_u} \vert}{\vert Q_{l_k}\vert}\Rr)$ and they are uniformly distributed conditioned the number. We use the similar argument for the expectation of $\tau_h u$, where we should study the case for particles in $\tau_{-h} Q_{l_u} \subset Q_{l_k}$
\begin{align*}
\E_{\rho, n}[\tau_h u]=& \sum_{m=0}^n \P_{\rho, n}[\mu(\tau_{-h} Q_{l_u})=m]\E_{\rho, n}[\tau_h u \vert \mu(\tau_{-h} Q_{l_u})=m] \\
=& \sum_{m=0}^n {n \choose m} \Ll(\frac{\vert \tau_{-h} Q_{l_u} \vert}{\vert Q_{l_k}\vert}\Rr)^m \Ll(1 - \frac{\vert \tau_{-h} Q_{l_u} \vert}{\vert Q_{l_k}\vert}\Rr)^{n-m} \\
& \qquad \qquad \times \fint_{(\tau_{-h} Q_{l_u})^m} \tilde{u}_m(x_1+h, \cdots x_m+h) \, \d x_1 \cdots \d x_m \\
=& \sum_{m=0}^n {n \choose m} \Ll(\frac{\vert Q_{l_u} \vert}{\vert Q_{l_k}\vert}\Rr)^m \Ll(1 - \frac{\vert Q_{l_u} \vert}{\vert Q_{l_k}\vert}\Rr)^{n-m} \fint_{(Q_{l_u})^m} \tilde{u}_m(x_1, \cdots x_m) \, \d x_1 \cdots \d x_m.
\end{align*}
Thus we establish $\E_{\rho, n}[\tau_h u] = \E_{\rho, n}[u]$ and $V$ has mean zero.
\end{itemize}

Now we can apply the lemma: for any $0 < \epsilon < \frac{1}{8}(\Vert u\Vert_{L^\infty} (l_k)^2)^{-1}$, we put $\sqrt{\A_{Q_{l_k}} g}/\E_{\rho, n}[\A_{Q_{l_k}}g]$ at the place of $f$ in \cref{eq:4.2.1} and combine with \cref{eq:4.2.2} to obtain that 
\begin{align*}
\E_{\rho, n}\Ll[\A_{Q_{l_k}}g (u- \tau_h u)\Rr] &\leq 2\epsilon\E_{\rho, n}[(u- \tau_h u) ((-\Delta_{Q_{l_k}})^{-1}(u- \tau_h u))]\E_{\rho, n}[\A_{Q_{l_k}}g]\\
& \qquad  + \frac{1}{\epsilon}\E_{\rho, n}\Ll[\sqrt{\A_{Q_{l_k}} g} \Ll((-\Delta_{Q_{l_k}}) \sqrt{\A_{Q_{l_k}} g}\Rr)\Rr].
\end{align*}
Notice that $(-\Delta_{Q_{l_k}})^{-1}:L^2 \to H^1$ well-defined thanks to the Lax-Milgram theorem and the spectral bound, we get 
\begin{multline}\label{eq:PtbSmall}
\E_{\rho, n}\Ll[\A_{Q_{l_k}}g (u- \tau_h u)\Rr] \\
\leq 8\epsilon (l_k)^2 \Vert u\Vert^2_{L^{\infty}}\E_{\rho, n}[\A_{Q_{l_k}}g]+ \frac{1}{\epsilon}\E_{\rho, n}\Ll[\sqrt{\A_{Q_{l_k}} g} \Ll((-\Delta_{Q_{l_k}}) \sqrt{\A_{Q_{l_k}} g}\Rr)\Rr].
\end{multline}
For the case $\epsilon > \frac{1}{8}(\Vert u\Vert_{L^\infty} (l_k)^2)^{-1}$, we have $1 \leq 8 \epsilon \Vert u\Vert_{L^\infty} (l_k)^2$, thus we use a trivial bound 
\begin{align}\label{eq:PtbBig}
\E_{\rho, n}\Ll[\A_{Q_{l_k}}g (u- \tau_h u)\Rr] \leq 2\Vert u\Vert_{L^\infty} \E_{\rho, n}\Ll[\A_{Q_{l_k}}g\Rr] \leq 16 \epsilon
(l_k)^2\Vert u\Vert^2_{L^\infty}\Ll[\A_{Q_{l_k}}g\Rr].
\end{align}
We combine \cref{eq:PtbSmall}, \cref{eq:PtbBig} and do optimization with for $\epsilon$ to obtain that 
\begin{align*}
\E_{\rho, n}\Ll[\A_{Q_{l_k}}g (u- \tau_h u)\Rr] \leq 4
l_k \Vert u\Vert_{L^\infty} \Ll(\E_{\rho, n}\Ll[\A_{Q_{l_k}}g\Rr]\E_{\rho, n}\Ll[\sqrt{\A_{Q_{l_k}} g} \Ll((-\Delta_{Q_{l_k}}) \sqrt{\A_{Q_{l_k}} g}\Rr)\Rr]\Rr)^{\frac{1}{2}}.
\end{align*}
Here the term $\E_{\rho, n}\Ll[\sqrt{\A_{Q_{l_k}} g} \Ll((-\Delta_{Q_{l_k}}) \sqrt{\A_{Q_{l_k}} g}\Rr)\Rr]$ is not the desired term and we should remove the conditional expectation here. For any $x \in Q_{l_k}$, using Cauchy-Schwartz inequality we have 
\begin{align*}
\A_{Q_{l_k}}\Ll(\frac{\vert\nabla g(\mu, x)\vert^2}{g(\mu)}\Rr) \A_{Q_{l_k}}g(\mu) \geq \Ll(\A_{Q_{l_k}}\vert \nabla g(\mu, x)\vert\Rr)^2 \geq \Ll\vert \A_{Q_{l_k}} \nabla g(\mu, x)\Rr\vert^2.
\end{align*}
Thus, in the term ${\E_{\rho, n}\Ll[\sqrt{\A_{Q_{l_k}} g} \Ll((-\Delta_{Q_{l_k}}) \sqrt{\A_{Q_{l_k}} g}\Rr)\Rr]}$ we have 
\begin{align*}
\E_{\rho, n}\Ll[\sqrt{\A_{Q_{l_k}} g} \Ll((-\Delta_{Q_{l_k}}) \sqrt{\A_{Q_{l_k}} g}\Rr)\Rr] &= \frac{1}{4}\E_{\rho, n}\Ll[\int_{Q_{l_k}} \frac{\Ll\vert \A_{Q_{l_k}} \nabla g(\mu, x)\Rr\vert^2}{\A_{Q_{l_k}}g(\mu)}\, \d \mu  \Rr]\\
&\leq \frac{1}{4}\E_{\rho, n}\Ll[\int_{Q_{l_k}} \A_{Q_{l_k}}\Ll(\frac{\vert\nabla g(\mu, x)\vert^2}{g(\mu)}\Rr)\, \d \mu  \Rr]\\
&= \E_{\rho, n}\Ll[\sqrt{g} \Ll((-\Delta_{Q_{l_k}}) \sqrt{g}\Rr)\Rr].
\end{align*}
Using the transpose invariant property for $\mu$, we obtain
\begin{align*}
\Ll\vert\E_{\rho, n}\Ll[\A_{Q_{l_k}}g (u- \tau_h u)\Rr]\Rr\vert \leq 4
l_k \Vert u\Vert_{L^\infty} \Ll(\E_{\rho, n}\Ll[\A_{Q_{l_k}}g\Rr]\E_{\rho, n}\Ll[\sqrt{g} \Ll((-\Delta_{Q_{l_k}}) \sqrt{g}\Rr)\Rr]\Rr)^{\frac{1}{2}},
\end{align*}
and put it back to \cref{eq:PerSetUp} and use Cauchy-Schwartz inequality 
\begin{align*}
&\Ll(\Er[g(u - \tau_h u)]\Rr)^2 \\
= &(l_k \Vert u\Vert_{L^\infty})^2\Ll(\sum_{n=0}^{\infty}\Pr[\mu(Q_{l_k}) = n]\Ll(\E_{\rho, n}\Ll[\A_{Q_{l_k}}g\Rr]\E_{\rho, n}\Ll[\sqrt{g} \Ll((-\Delta_{Q_{l_k}}) \sqrt{g}\Rr)\Rr]\Rr)^{\frac{1}{2}}\Rr)^2\\
\leq & (l_k \Vert u\Vert_{L^\infty} )^2\underbrace{\Ll(\sum_{n=0}^{\infty}\Pr[\mu(Q_{l_k}) = n] \E_{\rho, n}\Ll[\A_{Q_{l_k}}g\Rr]\Rr)}_{= \Er[\A_{Q_{l_k}}g]= \Er[g] = 1}  \Ll(\sum_{n=0}^{\infty}\Pr[\mu(Q_{l_k}) = n] \E_{\rho, n}\Ll[\sqrt{g} \Ll((-\Delta_{Q_{l_k}}) \sqrt{g}\Rr)\Rr]\Rr)\\
=& (l_k \Vert u\Vert_{L^\infty} )^2\Er[\sqrt{g}(-\Delta_{Q_{l_k}} \sqrt{g})].
\end{align*}
\end{proof}

\subsection{Entropy}

We recall the definition of $\delta$-good configuration for $\frac{L}{l} \in \mathbb{N}$
\begin{align*}
\mcl C_{L,l,\rho,\delta} = \Ll\{M \in \mathbb{N}^{\Ll(\frac{L}{l}\Rr)^d} \Ll\vert \forall 1 \leq i\leq \Ll(\frac{L}{l}\Rr)^d, \Ll\vert\frac{M_i}{\rho \vert Q_l \vert} - 1 \Rr\vert \leq \delta \Rr.\Rr\}.
\end{align*}
\begin{lemma}[Bound for entropy]\label{lem:Entropy}
Given $l \geq 1, \frac{L}{l} \in \mathbb{N}$, $0 < \delta < \frac{\rho}{2}$ for any $M \in \mcl C_{L,l,\rho,\delta}$, we have a bound for the entropy of $g_M$ defined in \cref{eq:defgM} that 
\begin{align}\label{eq:Entropy}
H(g_M) \leq C(d, \rho)\Ll(\frac{L}{l}\Rr)^d \Ll(\log(l) + l^d\delta^2\Rr). 
\end{align}
\begin{proof}
\begin{align*}
H(g_M) = \Er[g_M \log(g_M)] = -\Er[g_M \log(\Pr[\M_{L,l} = M])].
\end{align*}
It suffices to prove a upper bound for $-\log(\Pr[\M_{L,l} = M])$, which is 
\begin{align}\label{eq:LogPM}
-\log(\Pr[\M_{L,l} = M]) = -\log\Ll(\prod_{i=1}^{\Ll(\frac{L}{l}\Rr)^d} e^{-\rho \vert Q_l\vert}\frac{(\rho \vert Q_l\vert)^{M_i}}{M_i !}\Rr) = \sum_{i=1}^{\Ll(\frac{L}{l}\Rr)^d} -\log\Ll(e^{-\rho \vert Q_l\vert}\frac{(\rho \vert Q_l\vert)^{M_i}}{M_i !}\Rr).
\end{align}
For every term $M_i$, we set $\delta_i := \frac{M_i}{\rho \vert Q_l \vert} - 1$, and use Stirling's formula upper bound ${n! \leq e\sqrt{n}\Ll(\frac{n}{e}\Rr)^n}$ for any $n \in \mathbb{N}$
\begin{align*}
-\log\Ll(e^{-\rho \vert Q_l\vert}\frac{(\rho \vert Q_l\vert)^{M_i}}{M_i !}\Rr) & = \rho \vert Q_l\vert - M_i \log(\rho \vert Q_l\vert)+ \log(M_i !) \\
&\leq \rho \vert Q_l\vert - M_i \log(\rho \vert Q_l\vert)+ \log\Ll(e \sqrt{M_i}\Ll(\frac{M_i}{e}\Rr)^{M_i}\Rr) \\
&\leq \rho\vert Q_l\vert\Ll( \frac{M_i}{\rho\vert Q_l\vert} \log\Ll(\frac{M_i}{\rho\vert Q_l\vert}\Rr) + 1 - \frac{M_i}{\rho\vert Q_l\vert}\Rr) + \frac{1}{2}\log(M_i) \\
&= \rho\vert Q_l\vert\Ll( (1+\delta_i) \underbrace{\log\Ll(1+\delta_i\Rr)}_{\leq \delta_i} - \delta_i \Rr) + \frac{1}{2}\underbrace{\log(M_i)}_{\leq C\log(l)} \\
&\leq \rho\vert Q_l\vert(\delta_i)^2 + C\log(l). 
\end{align*}
We use $\vert \delta_i \vert \leq \delta$ and put it back to \cref{eq:LogPM} and obtain the desired result.
\end{proof}
\end{lemma}

\bigskip

\subsection*{Acknowledgments}
I am grateful to Jean-Christophe Mourrat for his suggestion to study this topic and inspiring discussions, Chenmin Sun and Jiangang Ying for helpful discussions. I would like to thank NYU Courant Institute for supporting the academic visit, where part of this project is carried out. 

\bibliographystyle{abbrv}
\bibliography{KawasakiRef}

\end{document}